 \documentclass[11pt,letterpaper]{amsart}

\usepackage{graphicx} 
\usepackage{fullpage}

\usepackage{amsmath,amssymb,mathrsfs} 
\usepackage{amsthm}
\usepackage[pagebackref,hypertexnames=false]{hyperref}
\usepackage{fancybox,fancyhdr,graphics,epsfig}
\usepackage[usenames,dvipsnames]{color}
\usepackage{bbm}
\usepackage{subcaption}
\usepackage{stmaryrd}

\def\covnum{N^{\mathrm{cov}}}

\def\pilim{\Pi_k^*}
\newcommand{\tocheck}[1]{#1}

 \usepackage{amsmath,amsthm,bbm,amssymb}
\DeclareMathOperator{\birth}{birth}
\DeclareMathOperator{\death}{death}

\DeclareMathOperator{\dist}{dist}

\DeclareMathOperator{\supp}{supp}

\DeclareMathOperator{\Vol}{Vol}

\DeclareMathOperator{\B}{B}

\DeclareMathOperator{\logg}{\log\log}

\DeclareMathOperator{\PH}{PH}

\def\R{\mathbb{R}}

\def\P{\mathbb{P}}

\def\S{\mathbb{S}}

\def\cA{\mathcal{A}}

\def\cC{\mathcal{C}}
\def\cD{\mathcal{D}}
\def\cF{\mathcal{F}}

\def\cH{\mathcal{H}}

\def\cN{\mathcal{N}}

\def\cP{\mathcal{P}}

\def\cR{\mathcal{R}}

\def\cE{\mathcal{E}}
\def\cX{\mathcal{X}}

\newcommand{\E}{\mathbb{E}} 

\newcommand{\mean}[1] {\E\left\{{#1}\right\}}

\newcommand{\ind}{\boldsymbol{\mathbbm{1}}} 
\newcommand{\indf}[1]{\ind\set{#1}} 

\newcommand{\set}[1]{\left\{#1\right\}}

\newcommand{\param}[1]{\left(#1\right)}
\newcommand{\abs}[1] {\left| {#1}\right|}

\newcommand{\ceil}[1] {\left\lceil{#1}\right\rceil}
\newcommand{\prob}[1]{\mathbb{P}\left(#1\right)}

\newcommand{\eps}{\varepsilon}

\newcommand{\rmin}{r_{\min}}

\newcommand{\be}{\mathbf{e}}

\providecommand{\setthms}[1]{#1}
\setthms{
\newtheorem{lem}{Lemma}[section]
\newtheorem{thm}[lem]{Theorem}
\newtheorem{prop}[lem]{Proposition}
\newtheorem{cor}[lem]{Corollary}

\newtheorem{rem}[lem]{Remark}

\theoremstyle{definition}
\newtheorem{defn}[lem]{Definition}

}

\def\Hg{\mathrm{H}}

\newcommand{\ninf}{n\to\infty}
\newcommand{\pois}[1]{\mathrm{Poisson}\param{{#1}}}

\newcommand{\fmax}{f_{\max}}

\newcommand{\limninf}{\lim_{\ninf}}

\newcommand{\bs}{\backslash}

\def\CC{\check{C}}
\def\const {C^*}
\definecolor{mygreen}{rgb}{0, 0.68, 0.31}
\definecolor{myred}{rgb}{1.0, 0,0}

\numberwithin{equation}{section}

\def\bsplit#1\esplit{\begin{split} #1 \end{split} }
\def\splitb#1\splite{\begin{split} #1 \end{split} }
\def\beq#1\eeq{\begin{equation} #1 \end{equation}}
\def\eqb#1\eqe{\begin{equation} #1 \end{equation}}

\title[Extreme  Cycles in Random \v Cech Complexes]{A Universal Law of Large Numbers \\for Extreme  Cycles in Random \v Cech Complexes}

\def\coef{c}
\def\fmaxD{f_{\max\!|_D}}
\def\fminD{f_{\min\!|_D}}

\DeclareMathOperator{\spt}{spt}

\DeclareMathOperator{\FillRad}{FillRad}
\DeclareMathOperator{\proj}{proj}
\DeclareMathOperator{\Cyl}{Cyl}
\DeclareMathOperator{\aff}{aff}
\DeclareMathOperator{\meshstar}{st}
\DeclareMathOperator{\Tan}{Tan}
\DeclareMathOperator{\diagram}{dgm}
\DeclareMathOperator{\PoissonProcess}{Poisson}

\newcommand{\coverconst}{\mu_k^{\mathrm{cov}}}
\newcommand{\flatnorm}[1]{\left\lVert#1\right\rVert_{\mathcal F}}
\newcommand{\geomreal}[1]{\left\llbracket #1 \right\rrbracket}
\newcommand{\card}[1]{\left|#1\right|}

\author{Omer Bobrowski\textsuperscript{1}}

\author{Primo\v{z} \v{S}kraba\textsuperscript{1,2}}
\address{\textsuperscript{1} School of Mathematical Sciences,
Queen Mary University of London,
London, United Kingdom}

\address{\textsuperscript{2} Faculty of Computer and Information Science,
University of Ljubljana,
Ljubljana, Slovenia}
\begin{document}

\begin{abstract}

We study the maximal multiplicative persistence of  $k$-cycles in random \v Cech complexes. Let $f:\R^d\rightarrow\R$ be a probability density function, let $\cP_n$ be a Poisson process with intensity $nf$, and let $\Pi_{k,n}$ denote the largest death-to-birth ratio among all non-essential $k$-cycles ($1\le k \le d-1$). 
For the uniform distribution in the unit hypercube, it was proved in \cite{bobrowskiMaximallyPersistentCycles2017} that $\Pi_{k,n} = \Theta\param{ \left(\frac{\log n}{\log\log n}\right)^{1/k}}$.
In this paper we sharpen and extend this result to a law of large numbers, for a  broad class of distributions. Most significantly, we show that the limiting constant depends \emph{only} on $d$ and $k$, and \emph{not} on the probability density $f$. Thus the extreme value of multiplicative persistence exhibits a \emph{universality} phenomenon. We show that the limiting constant is determined by the asymptotic normalized covering number of the $k$-dimensional sphere. 
Our proof identifies the geometric mechanism underlying maximal cycles,  a persistent isoperimetric inequality, which gives  sharp bounds on the number of points needed to generate a highly persistent cycle. By combining limiting covering number estimates with isoperimetric inequalities, we show that this minimum is asymptotically attained by efficient coverings of a $k$-sphere. A key ingredient is a geometric measure theory argument that uses compactness to relate discrete covering counts to the volume of a limiting cycle.
\end{abstract}

\maketitle

\section{Introduction}

Over the past two decades, \emph{persistent homology} \cite{edelsbrunner_topological_2002,zomorodian_computing_2005}
 has emerged as a versatile framework in geometry and topology, with applications to areas such as group theory and symplectic dynamics \cite{alpertConfigurationSpacesDisks2024,ellisPersistentHomologyGroups2011,entovLegendrianPersistenceModules2022,polterovichTopologicalPersistenceGeometry2020}, as well as to topological data analysis (TDA) \cite{carlssonTopologicalDataAnalysis2021,chazalHighdimensionalTopologicalData2017,deyComputationalTopologyData2022,edelsbrunnerComputationalTopologyIntroduction2010}. Given a filtration of topological spaces, it records the \emph{birth} and \emph{death} times of homological features as the filtration parameter increases. 
Thus, it provides a multiscale invariant that interpolates between local geometry and global topology. In recent years, persistent homology has become an important tool in data and network analysis, motivating a substantial body of work on the probabilistic behavior of persistence diagrams generated by random geometric complexes (see \cite{chazal2018density,hiraoka2018limit,kahle2011random,krebs2025asymptotic,owada2022convergence} and the survey in \cite{bobrowskiTopologyRandomGeometric2018}). 
In this paper we study the extremal behavior of persistent homology in random \v{C}ech complexes. 
Previous work in stochastic topology focused mainly on spatial extremes \cite{AdlerBobrowskiWeinberger2014Crackle,Owada2018BettiExtreme,OwadaAdler2017PointProcesses,OwadaBobrowski2020PersistenceCrackle}. Our goal here is to study extreme values for the lifetime of persistence cycles. 

Let $f:\R^d\rightarrow\R$ be a probability density function, and let $\mathcal P_n\subset\mathbb R^d$ be a Poisson process with intensity $nf$. We consider the associated \v Cech filtration, denoted  $\set{\CC_r(\mathcal P_n)}_{r\geq 0}$. For every class $\gamma$ in the $k$-th persistent homology $\PH_k(\cP_n)$, we denote its birth and death times by $\birth(\gamma)$ and $\death(\gamma)$, and its multiplicative persistence measure as
\[
\pi(\gamma):=\frac{\death(\gamma)}{\birth(\gamma)}.
\]
Our focus is on the largest persistence value appearing in dimension $k$, i.e., 
\[
\Pi_{k,n}:=\max_{\gamma\in\PH_k(\cP_n)}\pi(\gamma).
\]
The quantity $\Pi_{k,n}$ measures the size of the most prominent topological feature generated by the random complex. A natural question is therefore: 
{\bf How large can the most persistent $k$-cycle in a random geometric complex be?}
This question was first studied in \cite{bobrowskiMaximallyPersistentCycles2017}, where it was shown that for a homogeneous Poisson process in the unit cube (i.e. $f\equiv 1$), with high probability, we have
\[
A\Delta_{k,n}
\leq
\Pi_{k,n}
\leq
B\Delta_{k,n},
\]
where
\[
\Delta_{k,n}=
\left(
\frac{\log n}{\log\log n}
\right)^{1/k},
\]
and $0<A<B<\infty$ are some constants (that may depend on $d$ and $k$).
That result established the order of growth of the largest persistent cycle, revealing that extremal persistent homology is governed by rare geometric configurations that first appear  in the sparse regime. However, it left a few fundamental questions open:

\begin{enumerate}
\item Does $\Pi_{k,n}/\Delta_{k,n}$ converge to a deterministic limit (i.e., is there a law of large numbers)?
\item If so, what can we say about this limit, and the precise geometric configurations that generate it?
\item To what extent does this limiting behavior depend on the choice of $f$?
\end{enumerate}

The main goal of this paper is to answer these questions. Our main result shows that the largest $\death/\birth$ value (after normalization) converges in probability to an explicit constant depending only on the ambient dimension $d$ and the homological degree $k$, but not on $f$. In other words, $\Pi_{k,n}$ has a \emph{universal} limit.

A key observation is that the probabilistic limit is governed by a deterministic geometric problem, namely -- how many points are needed to generate a cycle with a prescribed persistence ratio? 
We establish an asymptotically sharp lower bound and show that efficient coverings of a $k$-sphere attain it. In this construction, the covering radius controls the birth time, while the sphere radius controls the death time. 
The lower bound combines covering estimates with isoperimetric inequalities, using a compactness argument from geometric measure theory to connect discrete covering counts with the mass of a limiting cycle. Consequently, we can show that the asymptotic behavior of $\Pi_{k,n}$ is governed by the covering number of the unit sphere $\S^k$.

Closest in spirit to our results is \cite{ChenavierHirsch2022Extremal}, where the authors prove a Poisson approximation for extremal $(\death-\birth)$ values. Their maximal lifetime results are for a homogeneous Poisson process and $k=d-1$, excluding critical percolation regimes. These restrictions (and in particular, lack of universality) arise since additive lifetimes are not scale-invariant by nature. In addition, \cite{ababneh2024maximal} studies the maximal death-to-birth ratio for random clique complexes, which have no underlying geometry. They identify its polynomial growth exponent, using fixed-size configurations (the boundary of a cross-polytope) to establish the lower bound.

The results presented here are also closely related to recent work on universality in random persistent homology. In \cite{bobrowskiUniversalNulldistributionTopological2023,bobrowskiUniversalityRandomPersistent2026,lim2026universal}, it was shown that the empirical distribution of the $\death/\birth$ persistence values converges to a universal limit (independent of the underlying sampling distribution). 
In this paper, we show that universality extends all the way to the extreme tail of the persistence diagram.

We note that our statements about extremal cycles are far from being a consequence of the results proved in \cite{bobrowskiUniversalityRandomPersistent2026}. In general, convergence of empirical measures does not imply convergence of extrema. Moreover, the persistence values appearing in a random persistence diagram may be highly dependent, which can strongly influence their extreme-value behavior. Therefore, the framework developed in \cite{bobrowskiUniversalityRandomPersistent2026} does not apply in the present setting. In particular, a key assumption there is additivity, that the $\Pi_{k,n}$ values clearly do not satisfy.

While this paper is theoretical in nature, the results we present also provide insight in the field of topological data analysis (TDA). 
Computing persistent homology on real-world datasets, typically produces a large number of cycles, where only a few of these cycles represent meaningful geometric or topological structures. Determining which features should be regarded as significant remains a fundamental challenge in TDA \cite{blumberg_robust_2013,chazal_geometric_2011,fasy_confidence_2014,reani_cycle_2023,vejdemo_johansson_multiple_2022}. Revealing the maximal persistence of cycles generated by randomness alone (i.e., ``topological noise"), our results can be used to develop new powerful tests to detect the meaningful structure in the data.

\section{Preliminaries}

\subsection{Persistent homology}\label{sec:homology}

Throughout the paper, all homology groups are taken with $\mathbb{Z}_2$ coefficients. While most of the steps apply to any field $\mathbb{F}$, optimality of the cycles can currently be shown only for $\mathbb{Z}_2$.
Given a filtration $\mathcal{F} = \{X_t\}_{t\in\R}$ of topological spaces (i.e.,  $X_{t_1}\subset X_{t_2}$ for $t_1\le t_2$), the  persistent homology of $\cF$, denoted $\PH_k(\cF)$,  is a module that consists of the vector spaces  
$\{\Hg_k(X_t)\}_{t\in \R}$ and the linear maps $\Hg_k(X_{t_1}) \rightarrow \Hg_k(X_{t_2})$ (for $t_1\le t_2$), induced by inclusion.
When $\Hg_k(X_t)$ is finite dimensional for every $t$ (as is the case in this paper), the module $\PH_k(\cF)$ can be decomposed into a sum of \emph{intervals}. These are rank-one sub-modules of the form $\mathbb{I}[b,d)$, that are non-trivial  inside the interval $[b,d)$ only (note that $d$ might be infinity). In other words, we have (see \cite{crawley-boeveyDecompositionPointwiseFinitedimensional2015} Theorem 1.1),
\[
\PH_k(\mathcal{F}) \cong \bigoplus\limits_{i} \mathbb{I}[b_i,d_i).
\]
A common way to summarize the information encoded in persistent homology is to look at the collection of $(b_i,d_i)$ pairs in the decomposition above. This is called the \emph{persistence diagram} of $\PH_k(\cF)$, and is denoted by $\diagram_k(\cF)$. See Figure \ref{fig:persistence} for an example.

\begin{figure}
    \centering
    \includegraphics[width=0.85\linewidth]{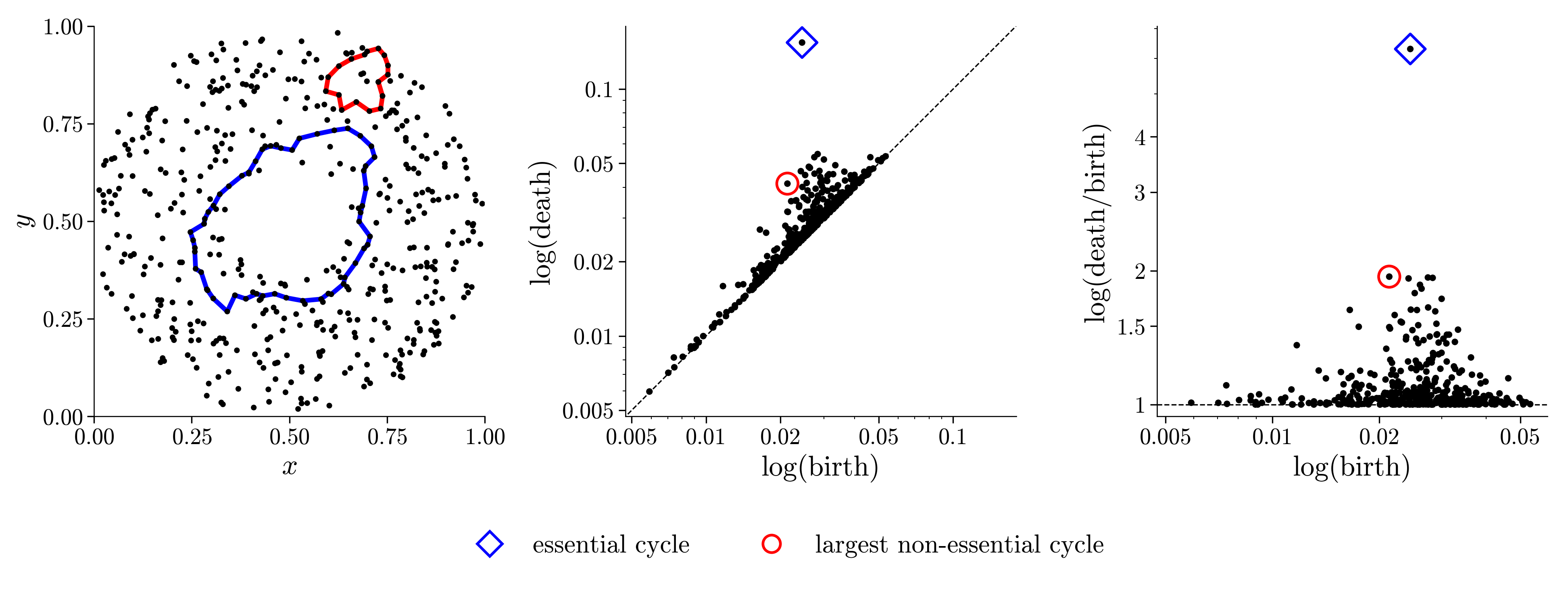}
    \caption{$\Hg_1$ persistence diagram. {\bf{(left)}} iid points sampled uniformly in an annulus. {\bf(middle)}  the persistence diagram for $\Hg_1$. Note that there is one essential cycle corresponding to the hole in the annulus, shown in blue {\bf(right)} transformed persistence diagram where the $y$-axis is $\death/\birth$ rather than $\death$. The  cycle marked in red is the non-essential cycle with the largest $\death/\birth$ value. This cycle is the subject of this paper. We note that all axes are in logarithmic scale to improve visualization.}
    \label{fig:persistence}
\end{figure}

\paragraph{\bf Essential vs.~non-essential cycles}
In this paper we study the persistent homology generated by points sampled from a subspace of $\R^d$. In this case, we  differentiate between two types of cycles. 
By \emph{essential} cycles we refer to persistent cycles that are generated by the geometry of the support, rather than the randomness of the points, and in particular, their death times are asymptotically a constant.
Note that in our setting, since all cycles are compactly supported in $\R^d$, the death time is always finite.
All other cycles are generated by the random placement of the points, and their scale decreases as we increase the number of points. We call these cycles \emph{non-essential}. The analysis in this paper applies to the non-essential cycles only, see Figure \ref{fig:persistence}. For more details and formal definition, see Section 2 in the Supplementary Information of \cite{bobrowskiUniversalNulldistributionTopological2023}.

\subsection{Filtrations of geometric complexes}

In this paper we focus on the persistent homology of filtrations generated by random point processes in $\R^d$. To impose structure on finite point-sets, we construct simplicial complexes using the \v Cech filtration as follows. 

Let $\cP\subset\R^d$ be a finite set. For $r\geq 0$, the \emph{\v{C}ech complex}, denoted $\CC_r(\cP)$, is an abstract simplicial complex with vertex set $\cP$, where a finite set $\sigma\subset\cP$ is a simplex if and only if $
\bigcap_{x\in\sigma} B_r(x)\neq\emptyset$. By the Nerve Lemma \cite{borsukImbeddingSystemsCompacta1948}, we have that
$\CC_r(\cP)\simeq B_r(\cP) := \bigcup_{p\in\cP}B_r(p)$, 
and therefore $\Hg_*(\CC_r(\cP))\cong \Hg_*(B_r(\cP))$.
We will use this equivalence throughout the paper.

A similar construction is the  \emph{Vietoris-Rips} complex, denoted $\cR_r(\cP)$, where the simplices are all subsets of points whose diameter is smaller than $r$. While we conjecture that similar results apply to the Vietoris-Rips complex, the proof for the  upper bound, and in particular the geometric statements in Section \ref{sec:geom} do not apply. Therefore, a different approach is required which will remain as future work.


\subsection{Multiplicative persistence}
A key motivation for using persistent homology in the context of data analysis  is that  $k$-cycles with a long lifetime 
are assumed to represent important topological features hidden in the data. In this paper, we consider the lifetime of cycles in a multiplicative way.
Let $\cF$ be a filtration, and let $\gamma\in \PH_k(\cF)$ for some $k>0$. We define the `$\pi$-value' of $\gamma$ as
\[
\pi(\gamma) := \frac{\death(\gamma)}{\birth(\gamma)}.
\]
This ratio provides a useful measure for the ``size'' of cycles in geometric filtrations, mainly as it is scale-invariant and robust to outliers. For more intuition and discussion see \cite{bobrowskiMaximallyPersistentCycles2017,bobrowskiUniversalNulldistributionTopological2023}.
Scale invariance also connects multiplicative persistence to an extremal geometric problem. It allows us to normalize the death radius of a  cycle to be $1$, so that a  birth time of at most $1/\pi_0$ gives a persistence value of at least $\pi_0$. In Section \ref{sec:geom}, we use this formulation to determine the minimum number of points needed to support such a persistent cycle.

\subsection{Poisson point processes}

Let $f:\mathbb R^d\rightarrow[0,\infty)$ be a probability density function. We write
\[
\mathcal P_n\sim \PoissonProcess(n;f)
\]
for the spatial Poisson  process on $\mathbb R^d$ with intensity
$nf(\cdot)$. By definition, $\cP_n$ satisfies the following.
\begin{itemize}
    \item For any Borel set $A\subset \R^d$, we have $\card{\cP_n \cap A} \sim \pois{nF(A)}$, where $F(A) = \int_A f(x)dx$.
    \item For any disjoint Borel sets $A,B\subset \R^d$, the random variables $\card{\cP_n\cap A}$ and $\card{\cP_n \cap B}$ are independent.
\end{itemize}
A simple way to generate a sample from $\cP_n$, is by taking a sequence of iid random variables $\set{X_1,X_2,\ldots}$ with density $f$, and then take $\cP_n = \set{X_1,\ldots, X_{N}}$ where $N\sim \pois{n}$ is a Poisson random variable.


\section{Main Results}\label{sec:main}

Our goal is to prove the limiting result for a wide class of densities $f$. We will focus on two main types of distributions.
Let $\cD_1$ be the collection of all  density functions $f:\R^d\rightarrow\R$ that have a compact support, such that $\fmax := \sup_{\R^d} f < \infty$, and  such that for all $x\in \supp(f)$ we have
\eqb\label{eqn:lb_vol}
    F(B_r(x)\cap \supp(f)) \ge C r^d,
\eqe
for some constant $C>0$, and for sufficiently small $r$.
Next, let $\cD_2$ be the collection of all density functions $f:\R^d\rightarrow\R$ of the form 
\[
f(x)=C_f\exp\left(-\coef\| x\|^q\right),
\] 
where $q,\coef>0$, $x\in\mathbb{R}^d$, and $C_f>0$ is a normalizing constant.
Our main results will be proven for the class $\cD = \cD_1\cup \cD_2$ that covers a wide range of distributions. 
Throughout, we use $F$ to denote the probability measure in $\R^d$, i.e., for any Borel set $A\subset\R^d$,
\[
    F(A) := \int_A f(x)dx.
\]

Let $\cP_n\sim\PoissonProcess(n;f)$ be a Poisson process in $\R^d$ with density $f\in \cD$. 
Let $\PH_{k,n}=\PH_k(\cP_n)$ be the  persistent homology generated by the \v Cech filtration $\set{\CC_r(\cP_n)}_{r\ge 0}$. Note that for class $\cD_1$, the persistence module might contain \emph{essential} cycles (see discussion in Section \ref{sec:homology}). We implicitly exclude these cycles from $\PH_{k,n}$ and assume it only consists of \emph{non-essential} cycles in the filtration. In particular, this implies that for large enough $n$, all the cycles  we consider in $\PH_{k,n}$ have a death-time smaller or equal to the coverage radius of the support (by definition), a fact that we will use in the proofs.

For every persistent cycle $\gamma\in\PH_{k,n}$ we define $\pi(\gamma) = \death(\gamma)/\birth(\gamma)$, and the maximal persistence,
\[
    \Pi_{k,n} := \max_{\gamma\in\PH_{k,n}} \pi(\gamma).
\] 
The scaling we will use is the same as in \cite{bobrowskiMaximallyPersistentCycles2017},
\[
    \Delta_{k,n} = \param{\frac{\log n}{\logg n}}^{1/k}.
\]
The limiting constant is
\begin{equation}\label{eqn:lim_cech}
\pilim = (\coverconst s_k (d/k-1))^{-1/k},
\end{equation}
where $\coverconst$ is the limiting covering constant (see Theorem \ref{thm:covering}), and $s_k$ is the volume of a unit sphere in $\R^{k+1}$.
Their product,   $\mu_k^{\mathrm{cov}}s_k$, is directly related to the minimum number of points needed to generate a cycle with large multiplicative persistence. More precisely, the minimum number of points needed for a cycle of persistence at least $\pi_0$ is asymptotically equal to $\mu_k^{\mathrm{cov}}s_k\pi_0^k$ as $\pi_0\rightarrow\infty$. We establish the lower bound in Proposition \ref{prop:optimal}, while efficient spherical coverings provide the matching construction.

Our main result is the following weak law of large numbers for the maximal persistence value $\Pi_{k,n}$.

\begin{thm}\label{thm:main}
Let $f\in \cD$, and let $\cP_n\sim \PoissonProcess(n;f)$. Let $1\le k \le d-1$, and let $\Pi_{k,n}$ be the maximal persistence value. Then,
\[
    \frac{\Pi_{k,n}}{\Delta_{k,n}} \xrightarrow{\P} \pilim,
\]
where $\pilim$ is defined in \eqref{eqn:lim_cech}, where $\xrightarrow{\P}$ refers to convergence in probability.
\end{thm}

\begin{rem}
Our proof of Theorem \ref{thm:main} below splits into two parts. The lower bound proof applies to any density function, not necessarily in $\cD$. The proof for the upper bound is where we use the assumptions on the class $\cD$. While we conjecture that Theorem \ref{thm:main} can be generalized far beyond $\cD$, it seems that the proofs will have to involve intricate  technical details on a case-by-case basis. We chose to focus on the two classes $\cD_1$ and $\cD_2$, where more general methods can be applied, and such that we cover distributions with both a compact and non-compact supports.
\end{rem}


\section{Covering Numbers and Optimal Cycles}\label{sec:geom}
Several parts of the proof of  Theorem \ref{thm:main} will require as input the minimal number of points to generate sufficiently persistent classes. 
In this section, we present a deterministic statement, arguing that the intuitive choice for a persistent cycle is indeed the one that minimizes the number of points. By ``intuitive choice" we refer to the spreading of points equally around a $k$-sphere, where the gaps between points control the birth time, and the sphere radius determines the death time (see Figure \ref{fig:eff}). 

\begin{defn}
For $A\subset \R^d$, let $\covnum_r(A)$ denote the covering number of $A$ with balls of radius $r$, i.e., 
\[
\covnum_r(A) := \min\left\{\card{\cX} : \cX\subset A\subseteq B_r(\cX) \right\},
\]
where $B_r(\cX)$ is the Euclidean ball cover of $\cX$ at radius $r$.
\end{defn}

A key ingredient we will use is the limit for the number of balls required to cover a set.
\begin{thm}[{\cite[Theorem~4]{AndersonEtAl2022}}]
\label{thm:covering}
Let $A\subset\R^d$ be compact and countably $k$-rectifiable,
with
$
0<\mathcal H^k(A)=\mathcal M^k(A)<\infty,
$
where $\cH^k$ is the Hausdorff measure, and $\mathcal M^k$ denotes the normalized
$k$-dimensional Minkowski content. Then
\[
\lim_{r\rightarrow 0 } r^k\covnum_r(A)
=
\coverconst\,\mathcal H^k(A),
\]
 where  $\coverconst>0$ depends only on $k$.
\end{thm}

The following proposition  will serve as an important part of the proof of Theorem \ref{thm:main}. 
\begin{prop}
\label{prop:optimal}
For  $1\le k < d$, let $N^{(k)}_{\pi_0}$  denote the smallest number of points in $\mathbb R^d$ needed
to generate a persistent  $\mathbb{Z}_2$ $k$-cycle $\gamma$  in the \v Cech filtration with $\pi(\gamma)\ge \pi_0$.  Then, for any $\delta>0$, for all sufficiently large  $\pi_0$,
\[  
N_{\pi_0}^{(k)}\geq(1-\delta)\pi_0^k \coverconst s_k,
\]
where $s_k$ is the volume of $\S^k$, and $\coverconst$ is the limiting covering constant from Theorem \ref{thm:covering}.
\end{prop}
This statement is precisely a persistent isoperimetric inequality for \v Cech filtrations. It lower bounds the number of points needed in terms of the required persistence (asymptotically). 
Note that this statement is entirely deterministic, and requires substantial geometric work, which is somewhat orthogonal to the main probabilistic arguments in this paper. We therefore postpone the proof to Section \ref{sec:optimal}. We remark  that the proof  makes extensive use of the fact that homology is taken over $\mathbb{Z}_2$, but we conjecture that the result should hold for any field $\mathbb{F}.$

\section{Proof of Theorem \ref{thm:main} -- Lower Bound}\label{sec:lower}

Fix $\eps>0$, and let $\pi_L = (1-\eps)\pilim \Delta_{k,n}$. Our goal in this section is to prove the following lemma.

\begin{lem}\label{lem:lower}
    Let $f:\R^d\rightarrow\R$ be any pdf. Then,
    \[
        \limninf \prob{\Pi_{k,n} < \pi_L} = 0.
    \]
\end{lem}   
To prove Lemma \ref{lem:lower} we construct a specific type of $k$-cycle with persistence at least $\pi_L$, which we call an
\emph{efficient configuration}. By ``efficient" we mean that these configurations aim to produce a large $\pi$-value with as few points as possible.
The key idea is to look for configurations of points that cover a $k$-dimensional sphere (see Figure \ref{fig:eff}).
The gaps between the points will control the birth time of the $k$-cycle, while the radius of the sphere will control the death time.
We will show that, as $n$ goes to infinity, with high probability there is at least one of these efficient configurations, with $\death/\birth$ ratio larger than $\pi_L$. 


We will consider a specific type of covering of the sphere $\S^k$ embedded in $\R^d$, which we call an ``efficient net".
Let $r>0$ and $\delta\in (0,1)$. An efficient $(r,\delta)$-net of $\S^k$ is a subset $\cN_{r,\delta}\subset \S^k$ that satisfies:
    \begin{enumerate}
        \item For all $x\in \S^k$, there is some $p\in\mathcal{N}$ such that $\|x-p\|<(1-\delta)r$.
	\item The balls $B_{\delta r}(p)$ and $B_{\delta r}(p')$ are disjoint for any $p,p'\in \cN_{r,\delta}$.
    \end{enumerate}
We can construct $\cN_{r,\delta}$ as follows. Let $\cA_{r,\delta}$ be the centers of a minimal ball-covering of $\S^k$ at radius $(1-3\delta)r$. 
Pick a point $p\in \cA_{r,\delta}$, and remove all $p'\in\cA_{r,\delta}\cap B_{2\delta r}(p)$ ($p'\ne p$).
Repeat this process until nothing else can be removed. By the triangle inequality, this results in an efficient $(r,\delta)$-net. Furthermore, as an efficient net is a constrained covering, we have
\eqb\label{eqn:net_bound}
\covnum_{(1-\delta)r} (\S^k)\leq |\mathcal{N}_{\delta,r}|\leq \covnum_{(1-3\delta)r}(\S^k).
\eqe
\begin{defn}[Efficient configurations]\label{def:eff_config}
	Let $\cP\subset \R^d$ be a finite set, $r$ and $\delta\in(0,1)$. Let $\S^k$ be the unit $k$-dimensional sphere embedded in the first $k+1$ coordinates of $\R^d$, centered at the origin. Let $\cN_{r,\delta}$ be an efficient $(r, \delta)$-net of  $\S^k$. 
	We say that $\cP$ contains an efficient configuration if the following holds:
\begin{enumerate}
    \item 
    $|B_{\delta r}(p)\cap \cP| = 1$ for all $p\in \mathcal{N}_{r,\delta}$.
    
    \item Aside from the points in (1), there are no other points from $\cP$ lying in $Q:= [-3,3]^d$, i.e., 
    \[
    \cP \cap (Q\bs B_{\delta r}(\cN_{r,\delta})) = \emptyset.
\]
\end{enumerate}
\end{defn}

Figure \ref{fig:eff} shows an example of an efficient configuration for $k=1$.


\begin{lem}\label{lem:efficient}
Let $r < 1/2$. If $\cP$  contains an efficient configuration, then there is a $k$-cycle  $\gamma\in \PH_k(\cP)$ such that $$\pi(\gamma) \ge1/r-\delta.$$ 
\end{lem}

\begin{proof}

For this proof we will use the Nerve Lemma mentioned earlier, and consider the homology of $B_r(\cP)$ as a proxy for $\cC_r(\cP)$.
Let $\cX\subset\cP$ be an efficient configuration, then at radius $r$, we have $\S^k \subset B_r(\cX) \subset B_{r(1+\delta)}(\S^k)$ by the properties of the chosen net. Condition (2) in Definition \ref{def:eff_config} guarantees that this cover is indeed a non-trivial $k$-cycle, whose death time is at least $1-\delta r$. Thus, we have a $k$-cycle $\gamma$ with $\pi(\gamma)\ge (1-\delta r)/r$, completing the proof.
\end{proof}
\begin{figure}[htbp]
    \centering
\includegraphics[width=0.4\textwidth,page=19]{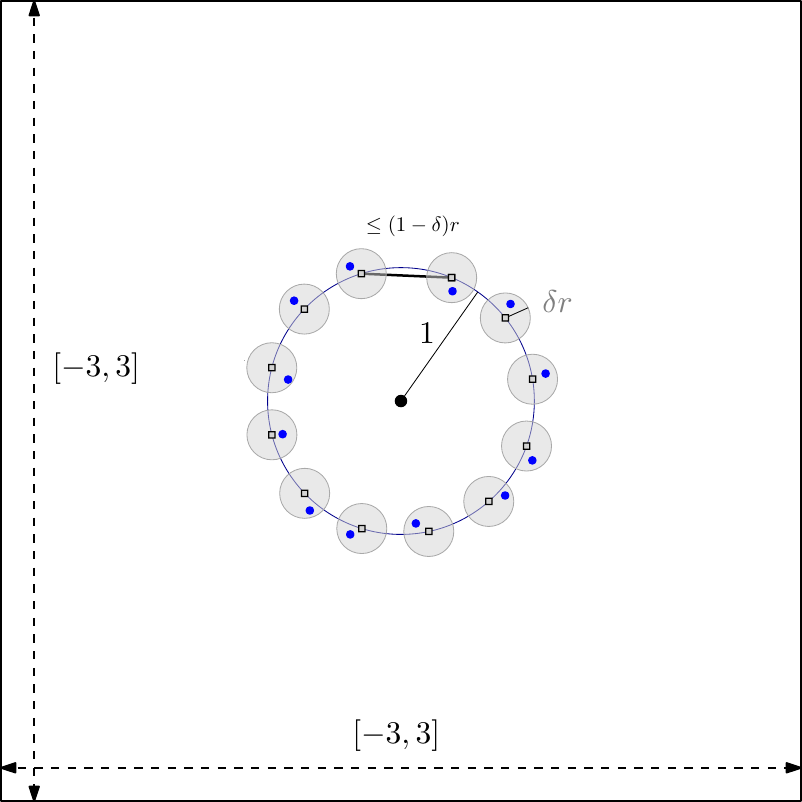}
    \caption{An example of an efficient configuration for $k=1$ centered at the origin. The points are arranged near a circle, are well-separated, and there are no other points in $[-3,3]^2$ box. The birth radius here will be less than $r$, death radius more than $1-\delta r$ and therefore $\pi(\gamma)\ge 1/r-\delta$. }
    \label{fig:eff}
\end{figure}

\begin{proof}[Proof of Lemma \ref{lem:lower}]
    Let $f:\R^d\rightarrow\R$ be a probability density function.
    There exists a closed box
    $D\subset \supp(f)$ with a positive measure, such that $\fminD:=\mathrm{ess\,inf}_{x\in D}f(x)>0$ and $\fmaxD:=\mathrm{ess\,sup}_{x\in D}f(x)<\infty$. 
   Without loss of generality, we will assume that $D$ has a side-length of $1$ (otherwise, we can scale all the arguments below by the  side-length of $D$, which will not affect the $\pi$-values).
   
   Take $R$ to be such that $nR^d = \alpha\log n$, where $\alpha>0$ will be set later. Next, take a grid of size $6R$ in $\R^d$, and let $D_1,D_2,\ldots, D_N$ be the grid boxes that are contained in $D$ (so that $\fminD \le f(x)\le \fmaxD$ for all $x\in D_i$). The number of these boxes is then $N= (6R)^{-d}+ o(R^{-d}) = \Theta(n/\log n)$. Setting $r=R/\pi_L$, our goal is to show that at least one of the $D_i$'s contains an efficient configuration in the sense of Definition \ref{def:eff_config}  (after scaling by $R$). This will imply that we have a cycle $\gamma \in \PH_k(\cP_n)$ with 
   \[
   \pi(\gamma) \ge R/r-\delta = \pi_L -  \delta.
   \]
  Since $\pi_L\rightarrow \infty$, we can repeat the same argument replacing $\eps$ with $\eps/2$, which then implies that for large enough $n$, $\pi(\gamma)\ge \pi_L$.
  
Let $\delta>0$, and let $\cN_{r,\delta}^{(i)}$ denote the efficient net and $\cN_{r/R,\delta}$ denote the net after scaling by $R$ and translating to the center of $D_i$.
Let $\cE_i$ be the event that $D_i$ contains an efficient configuration (scaled by $R$), then
\[
	\prob{\cE_i}= e^{-nF(D_i)} \prod_{p\in \cN_{r,\delta}^{(i)}}  nF(B_{\delta r}(p)) \ge e^{-n\fmaxD (6R)^d} (n \fminD \omega_d (\delta r)^d)^{|\cN_{r,\delta}^{(i)}|},
\]   
where $\omega_d$ is the volume of a $d$-dimensional unit ball. 
From Theorem \ref{thm:covering}, if we take $\delta<\eps/4$, then for $n$ large enough,
\[
\covnum_{(1-3\delta)/\pi_L}(\S^k) \cdot \param{\frac{1-3\delta}{\pi_L}}^k \le \param{\frac{1-4\delta}{1-\eps}}^k s_k \coverconst.
\]
From \eqref{eqn:net_bound} and \eqref{eqn:lim_cech}, this implies that
\begin{equation}\label{eqn:M}
  |\cN_{r,\delta}^{(i)}|\le  \covnum_{(1-3\delta)/\pi_L}(\S^k) \le \param{\frac{1-4\delta}{1-3\delta}}^k \param{\frac{\pi_L}{1-\eps}}^k s_k \coverconst = (1-\delta')^k \frac{k}{d-k}\cdot \frac{\log n}{\logg n},
\end{equation}
where $\delta' = \delta/(1-3\delta)$.
Denoting the last term as $M$, and using  the fact that $r=R/\pi_L$ we have
\[
\prob{\cE_i}\ge  ( C_1 n (R/\pi_L)^d)^{M} e^{-nC_2R^d},
\]
where $C_1=\fminD \omega_d \delta^d$ and $C_2 = \fmaxD 6^d$, and we used the facts that  $\pi_L \propto \param{\frac{\log n}{\logg n}}^{1/k}$ and $n R^d = O(\log n)$, implying that $
n(R/\pi_L)^d \rightarrow 0$.
Taking the logarithm, and using \eqref{eqn:M},
\[
\splitb
\log(\prob{\cE_i}) &\ge M(\log(C_1)+\log(nR^d)-d\log(\pi_L))  -nC_2R^d\\
&=  M\param{C_3 +\logg(n)-\frac{d}{k}\log\param{\frac{\log n}{\logg n}}}  - \alpha C_2 \log n\\
&= -\log n ((1-\delta')^k+\alpha C_2) + o(\log n),
\splite
\]
for some $C_3>0$.
Therefore,
\[
\prob{\bigcap_i \cE_i^c} \le (1- n^{-((1-\delta')^k+\alpha C_2)})^N \le e^{-N n^{-(1-\delta')^k-\alpha C_2}}\le e^{-C_4 n^{1-(1-\delta')^k-\alpha C_2} / \log n},
\]
for some $C_4>0$, and we used the fact that $N=\Theta(n/\log n)$.
Taking $\alpha$ to be sufficiently small we then get
\[
\prob{\bigcap_i \cE_i^c} \rightarrow 0,
\]
implying there must exist a $k$-cycle with $\pi(\gamma)\ge \pi_L$. This completes the proof.

\end{proof}

\section{Proof of Theorem \ref{thm:main} -- Upper Bound}
\label{sec:upper}

In this section, we  let $\pi_U = (1+\eps)\pilim \Delta_{k,n}$. Our goal is to prove the following lemma.

\begin{lem}\label{lem:upper}
    Let $f\in \cD$. Then,
    \[
        \limninf \prob{\Pi_{k,n} > \pi_U} = 0.
    \]
\end{lem}   

The main idea behind the proof of Lemma \ref{lem:upper} is to  establish the smallest number of points required to construct a cycle of persistence $\pi_U$. Then, by lower bounding the required birth radius, we show that such configurations are too large to appear at the required birth radius.

We will use a method related to the idea of nonuniform random geometric graphs studied in \cite{iyerNonuniformRandomGeometric2012a}. The main idea is to introduce a notion of connectivity and coverage of $B_r(\cP_n)$, where instead of considering balls of a fixed radius, we take a radius that is adapted to the density at every point.

\subsection{Adaptive connectivity}
Let $x\in \R^d$ be such that $f(x)>0$, and define $r_n(x)$ via
\eqb\label{eqn:adapt_r}
nr_n^d(x) := \frac{a(\log n)^{-1/\beta}}{f(x)},
\eqe
where $a,\beta>0$ will be determined later. If $f(x)=0$ we define $r_n(x) = \infty$, but this will have no effect on the results.
Let $\cX\subset \R^d$ be finite, and define
\[
    x_* := \arg\max_{\cX} f(x),\quad r_* := r_n(x_*) = \arg\min_{\cX} r_n(x).
\]
Later, we will be interested to check whether there are sets $\cX$ for which $B_{r_*}(\cX)$ is connected. The following lemma addresses that.

\begin{lem}\label{lem:ad_con}
    Let $D\subset \R^d$ be a Borel set, and $m\ge 2$. 
    Then, for sufficiently small values of $a$, we have
    \[
        \prob{\exists \cX \in (\cP_n\cap D)^{(m)} : B_{r_*}(\cX)\text{ is connected} } \le \frac{en}{m} F(D)\param{\frac12(\log n)^{-1/\beta}}^{m-1},
    \]
where $\cP^{(m)}$ is the collection of all subsets of size $m$ of $\cP$, and $F(D) = \int_D f(z)dz$.
\end{lem}

\begin{proof}
The proof is graph-theoretic, and therefore we will use the random geometric graph $G_r(\cX)$, where two vertices $x,y\in\cX$ are connected if and only if $\|x-y\|\le 2r$. Note that  $G_r(\cX)$ is connected if and only if $B_r(\cX)$ is.

Let $N_m(D)$ be the number of subsets  $\cX\subset\cP_n\cap D$ with $m$ vertices such that $G_{r_*}(\cX)$ is connected. Then
\[
N_m(D) = \sum_{\cX\in \cP_n^{(m)}}\mathbbm{1}\{G_{r_*}(\cX) \text{ is connected}\}\indf{\cX\subset D}, 
\]
and from Mecke's formula (see, e.g., \cite{penroseRandomGeometricGraphs2003}), we have
\[
\mean{N_m(D)}=\frac{n^m}{m!}
\prob{G_{r_*}(\cX')\text{ is connected}, \cX'\subset D},
\]
where $\cX'=\set{X_1,\ldots,X_m}$ is a collection of $i.i.d.$ variables with density $f$. If $G_{r_*}(\cX')$ is connected, then it must contain a subgraph isomorphic to a tree on $m$ vertices. Suppose that $\Gamma$ is a labeled tree on the vertices $\{1,...,m\}$. Without loss of generality, we assume that $x_* = x_1$, and we also take $x_1$ to be the root of $\Gamma$.
The fact that $x_*=x_1$ implies that for all $2\le i\le m$ we have $f(x_i)\le f(x_1)$. In addition, for all $2\le i\le m$ we define $\text{par}(i)$ as the parent of vertex $i$ in the tree. Suppose also that the vertices are ordered so that par$(i)<i$. If $G_{r_*}(\cX')$ contains a subgraph isomorphic to $\Gamma$, then every $X_i\in \cX'$ must be connected to $X_{\text{par}(i)}$, which implies that $X_i\in B_{2r_*}\left(X_{\text{par}(i)}\right)$. Therefore, 
    \[
    \splitb
        \prob{G_{r_*}(\cX') \text{ contains } \Gamma, \cX' \subset D} &\le m\int_{D}f(x_1)\int_{B_{2r_n(x_1)}(x_{\mathrm{par}(2)})}f(x_2)\indf{f(x_2)\le f(x_1)}\\
        &\qquad\cdots \int_{B_{2r_n(x_1)}(x_{\mathrm{par}(m)})} f(x_m)\indf{f(x_m)\le f(x_1)}dx_m\cdots dx_1,
    \splite
    \]
where we used the assumption that $r_*=r_n(x_1)$, and compensated for this assumption by multiplying the integral by $m$.
Next, using \eqref{eqn:adapt_r}, we have
    \[
\int_{B_{2r_n(x_1)}(x_{\mathrm{par}(i)})}f(x_i)\indf{f(x_i)\le f(x_1)} \le 2^d\omega_df(x_1) r_n^d(x_1)  = 2^d\omega_d an^{-1}(\log n)^{-1/\beta}.
    \]
Cayley's formula gives us that there are exactly $m^{m-2}$ labeled trees on $m$ vertices, and therefore,
\[
\splitb
\mean{N_m(D)} &\le \frac{n^m}{m!} m^{m-1} F(D) \param{2^d \omega_d a n^{-1}(\log n)^{-1/\beta}}^{m-1}\\
&\le \frac{en}{m}F(D)\param{2^d\omega_d ae (\log n)^{-1/\beta}}^{m-1},
\splite
\]
where we used the Stirling approximation.
Choosing $a$ small enough so that $2^d\omega_d a e < 1/2$, we have
\[
\mean{N_m(D)}\le \frac{en}{m} F(D)\param{\frac12(\log n)^{-1/\beta}}^{m-1}.
\]
Together with Markov's inequality, this completes the proof.
\end{proof}
The following special cases will be useful for us.
\begin{cor}\label{cor:ad_con}
Take $D=\R^d$, and $M_n = \ceil{\beta\frac{\log n}{\logg n}}$. Then
\[
        \prob{\exists \cX \in \cP_n^{(M_n)} : B_{r_*}(\cX)\text{ is connected} }\rightarrow 0.
\]
Furthermore, taking $\rmin = \inf_{\R^d}r_n(x)$, then 
\[
        \prob{\exists \cX \in \cP_n^{(M_n)} : B_{\rmin}(\cX)\text{ is connected} }\rightarrow 0.
\]
\end{cor}

\begin{proof}
 Lemma \ref{lem:ad_con}, with $F(D)=1$ yields, 
\[
 P:=\prob{\exists \cX\in \cP_n^{(M_n)} : B_{r_*}(\cX)\text{ is connected}} \le \frac{en}{M_n}\param{\frac12(\log n)^{-1/\beta}}^{M_n-1}.
\]
Taking  logarithms of both sides, we have
\[
    \log P \le  1 + \log n - \frac{M_n-1}{\beta}\logg n - (M_n-1)\log 2 - \log M_n  = -M_n\log 2  + o(M_n) \rightarrow -\infty,
\]
and therefore $P\rightarrow 0$. Finally, since for any subset $\cX$ we have $\rmin \le r_*$, we have $B_{\rmin}(\cX)\subset B_{r_*}(\cX)$ and this concludes the proof.
\end{proof}

\begin{cor}\label{cor:ad_small_comp}
Let $f(x) = C_fe^{-\|x\|^q}$, and take $L_n^q=\log n - C\logg n$ for any $C>0$. Take $D_n = \R^d \bs B_{L_n}(0)$, then there exists a fixed $m_0>0$ such that 
\[
\prob{\exists \cX\in (\cP_n \cap D_n)^{(m_0)}: B_{r_*}(\cX)\text{ is connected}} \rightarrow 0.
\]
\end{cor}

\begin{proof}
    First, note that
\[
	F(D_n) = C_f \int_{D_n} e^{-\|x\|^q} dp = 
    C_1 \int_{L_n^q}^\infty t^{d/q-1}e^{-t} dt.
\]
By properties of the incomplete gamma function, since $L_n\rightarrow\infty$, we have
\[
  F(D_n)\approx C_2 L_n^{d-q}e^{-L_n^q} = n^{-1}(\log n)^{d/q-1+C}.
\]
Therefore, from Lemma \ref{lem:ad_con}
\[
\prob{\exists \cX\in (\cP_n \cap D_n)^{(m_0)}: B_{r_*}(\cX)\text{ is connected}}\le \frac{e}{m_0} (\log n)^{d/q-1+C} \param{\frac12(\log n)^{-1/\beta}}^{m_0-1},
\]
and for $m_0$ large enough (but fixed), the last term goes to zero.
\end{proof}
\subsection{Adaptive coverage}
Let $x\in \R^d$, and define $R_n(x)$ via
\eqb\label{eqn:adapt_R}
    n R_n^d(x) := \frac{A\log n}{f(x)},
\eqe
for some $A>0$, which will be set later. If $f(x)= 0$ we set $R_n(x) = \infty$.

The next lemma is about ``adaptive coverage", which means that every point in $\R^d$ is covered by a ball of radius $R_n(x)$ around the point process $\cP_n$. To this end, we define the event
\[
\cA_n := \set{\forall x\in \R^d : \dist(x,\cP_n) \le R_n(x)}.
\]
\begin{lem}\label{lem:ad_cov}
Let $f(x) = C_fe^{-\|x\|^q}$. Then there exists a value of $A$, such that
\[
	\limninf\prob{\cA_n} = 1.
\]
\end{lem}

\begin{proof}
We split the proof into two cases.

\noindent {\underline{Case I:}} Let $x\in \R^d$ be such that $R_n(x) \le 2\|x\|$. Define the ball
\[
B_n(x) := B_{R_n(x)/2}\param{x- \frac{R_n(x)}{2}\cdot\frac{x}{\|x\|}} .
\]
Note that for all $z\in B_n(x)$, we have $\|z\| \le \|x\|$, which implies that $f(z) \ge f(x)$.
Also, we have $B_n(x) \subset B_{R_n(x)}(x)$. Therefore, using \eqref{eqn:adapt_R} we have
\[
F(B_n(x)) \ge \omega_d (R_n(x)/2)^d f(x) = \omega_d 2^{-d} A n^{-1} \log n,
\]
and
\[
\prob{B_n(x)\cap \cP_n = \emptyset} = e^{-nF(B_n(x))} \le n^{-c A},
\]
for $c = \omega_d 2^{-d}$.

Next, note that $R_n(x)\le 2\|x\|$, and using \eqref{eqn:adapt_R},
\[
	  2^{-d} A\frac{\log n}{n f(x)} = \frac{A}{2^dC_f}\frac{\log n}{n} e^{\|x\|^q} \le \|x\|^d ,
\]
and this implies that $\|x\| \le 2 (\log n)^{1/q}$.
Let $S_n$ be a $\delta_n$-net of the set $\cR_n :=\set{x\in \R^d : R_n(x)\le 2\|x\|}$,
with $\delta_n = \frac{R_n(0)}{\logg n}$. Then,
\[
	|S_n| \le C_1 \param{\frac{(\log n)^{1/q}\logg n}{R_n(0)}}^{d} = O\param{ n (\log n)^{d/q-1} (\logg n)^d},
\]
and 
\[
	\prob{\exists s\in S_n : B_{R_n(s)}(s) \cap \cP_n = \emptyset} \le C_2 n^{1-cA} (\log n)^{d/q-1} (\logg n)^d.
\]
Taking $A$ sufficiently large, the last probability goes to zero, indicating that all points in the net $S_n$ are adaptively covered. 
It remains to show that all other points in $\cR_n$
are also adaptively covered.

Define $R_n'(x) = (A'/A)^{1/d}R_n(x)$, where $A' < A$ is chosen so that we still have
\[
\prob{\exists s\in S_n : B_{R_n'(s)}(s)\cap \cP_n =\emptyset} \rightarrow 0.
\]
Therefore, for every point $x\in \cR_n$ we can write
\[
\dist(x, \cP_n) \le \|x-s_x\| + \dist(s_x,\cP_n)\le \|x-s_x\| + R_n'(s_x),
\]
where $s_x$ is the nearest neighbor of $x$ in the net $S_n$, and thus  $\|x-s_x\| \le \delta_n$. 

If $0<q\le 1$, this immediately implies that 
\[
\abs{\|x\|^q-\|s_x\|^q} \le \delta_n^q\rightarrow 0.
\]
If $q>1$, we can use the inequality $|a^q-b^q| \le q \max(a,b)^{q-1}|a-b|$, and since $\|x\|,\|s_x\| \le 2(\log n)^{1/q}$, we have
\[
\abs{\|x\|^q-\|s_x\|^q} \le q 2^{q-1}(\log n)^{1-1/q}\delta_n \rightarrow 0.
\]
In other words, for all $q$ we can find $\eta_n\rightarrow 0$ such that $\abs{\|x\|^q-\|s_x\|^q}\le \eta_n$.
Therefore, 
\[
\frac{R_n'(s_x)}{R_n'(x)} = \param{\frac{f(x)}{f(s_x)}}^{1/d} = e^{(\|s_x\|^q-\|x\|^q)/d} \le e^{\eta_n/d}.
\]
Thus, we have
\[
\dist(x,\cP_n) \le \delta_n + (1+\eta_n)R_n'(x),
\]
and for large enough $n$ the last term is smaller than $R_n(x)$ for all $x$. 
Thus, we showed that $\dist(x,\cP_n)\le R_n(x)$ for all $x$ with $R_n(x) \le 2\|x\|$.

\noindent {\underline{Case II:}}
Let $x\in \R^d$ such that $R_n(x) > 2\|x\|$.
Fix $c>0$, and note that for large enough $n$,
\[
n F(B_{cR_n(0)}(0)) \ge f(1)\omega_d c^d n R_n(0)^d = e^{-1}\omega_d c^d A \log n,
\]
where we used the fact that $R_n(0)\rightarrow 0$.
Therefore $\prob{B_{cR_n(0)}(0)\cap\cP_n = \emptyset} \rightarrow 0$, and so with high probability we can assume there exists a point $p\in B_{cR_n(0)}(0)\cap \cP_n$. 

Since $R_n(x)\ge R_n(0)$, we have
\[
\|p\| \le cR_n(0) \le cR_n(x),
\]
and therefore for every $x$ with $R_n(x)>2\|x\|$, we have
\[
\|x-p\| \le \|x\| + \|p\| \le (1/2+c)R_n(x).
\]  
Taking $c<1/2$ we have $\|x-p\| \le R_n(x)$ implying that $\dist(x,\cP_n) \le R_n(x)$.
This completes the proof.
\end{proof}

\begin{cor}\label{cor:ad_cov}
Let $f(x) = C_f e^{-\|x\|^q}$, and define $R_n(L) := R_n(L \be_1)$, $\be_1 = (1,0,\ldots,0)$. Then 
    \[
        \limninf\prob{\forall L>0: B_L(0) \subset B_{R_n(L)}(\cP_n)}= 1.
    \]
\end{cor}

\begin{proof}
For every $L>0$ and $x\in B_L(0)$, we have $f(x)\ge f(L)$, and therefore $R_n(x)\le R_n(L)$. From Lemma \ref{lem:ad_cov}, we can therefore conclude that $\dist(x,\cP_n) \le R_n(L)$, which means that $x\in B_{R(L)}(\cP_n)$, completing the proof.
\end{proof}

We are now ready to prove Lemma \ref{lem:upper}. The main idea in the proof is to show that there are no cycles of persistence greater than $\pi_U$,  merely because such cycles require very large components to exist in a very sparse regime.
We split the proof between the two classes $\cD_1$ and $\cD_2$.

\subsection{Proving Lemma \ref{lem:upper}, class $\cD_1$}

\begin{proof}[Proof  of Lemma \ref{lem:upper} -- the $\cD_1$ class]
Recall that class $\cD_1$ contains  bounded probability density functions, whose support is a compact subset of $\mathbb{R}^d$, and are bounded away from zero (on their support).
Suppose there exists $\gamma\in \PH_{k,n}$ such that $\pi(\gamma)>\pi_U$. Let $M_n$ be the smallest number of points needed to generate $\gamma$. By Proposition \ref{prop:optimal}, we have
\[
M_n \ge (1-\delta)\pi_U^k \coverconst s_k = (1+\eps') \frac{k}{d-k}\frac{\log n}{\logg n},
\]
where we used the fact that $\pi_U^k = (1+\epsilon)^k (\Pi_k^*)^k\frac{\log n}{\log\log n}$, and where $\eps' = (1-\delta)(1+\eps)^k -1$, assuming $\delta$ is sufficiently small.
Set $\beta=(1+\epsilon')\frac{k}{d-k}$, then from Corollary \ref{cor:ad_con}, if we take $nr^d = a\fmax^{-1}(\log n)^{-1/\beta}$, then $B_r(\cP_n)$ has no connected components of size larger than $\ceil{\beta\frac{\log n}{\log\log n}}$ (with high probability). By the Nerve Lemma, this implies that  if  $\birth(\gamma) =r$ then $nr^d > a\fmax^{-1}(\log n)^{-1/\beta}$.
Furthermore, condition \eqref{eqn:lb_vol} implies (cf.~\cite{cuevas2004boundary}) that with high probability $\supp(f)\subset B_R(\cP_n)$ for  $nR^d>C\log n$, for some constant $C>0$. Therefore, with high probability, for any persistent $k$-cycle $\gamma\in\PH_{k,n}$, if $\death(\gamma) =R$, then necessarily  $nR^d\leq C\log n$. 

Thus, with high probability, for any such $\gamma$, we have
\[
\pi(\gamma) = \frac{R}r=\left(\frac{nR^d}{nr^d}\right)^{1/d} \leq \left(\frac{C\log n}{a\fmax^{-1}(\log n)^{-1/\beta}}\right)^{1/d}= O\param{ \left((\log n)^{1+1/\beta}\right)^{1/d}}.
\]
Since we took $\beta = (1+\eps')\frac{k}{d-k}$, we have $1+1/\beta < \frac{d}{k}$, and therefore for all $\gamma\in\PH_{k,n}$ we have
$
\pi(\gamma)< \pi_U$ (for large enough $n$),
which leads to a contradiction.
This concludes the proof.
\end{proof}

\subsection{Proving Lemma \ref{lem:upper}, class $\cD_2$}

\begin{proof}[Proof of Lemma \ref{lem:upper}]
Recall that here $f(x) =C_fe^{-c\|x\|^q}$. Note that since $(\death/\birth)$ is scale invariant, we can assume without loss of generality that $c=1$.

Suppose that there exists a persistent $k$-cycle, with $\birth(\gamma)=r_\gamma$, and with $\pi(\gamma) \geq \pi_U$. Let $\cP_\gamma$ represent the set of points in $\cP_n$ that forms a generator of $\gamma$ at radius $r_\gamma$. Similarly to the $\cD_1$ case, we can show that $|\cP_\gamma| \ge M_n = \ceil{\beta\frac{\log n}{\logg n}}$.
From Corollary \ref{cor:ad_con} we therefore conclude that $r_\gamma > r_n(0)$.

Next, note that $f(L)$ is decreasing in $L$, therefore we can find a value of $L_\gamma$, such that
\[
    f(L_\gamma)nr_\gamma^d = a(\log n)^{-1/\beta},
\]
or in other words, a value of $L_\gamma$ such that $r_n(L_\gamma) = r_\gamma$. Also define $D_\gamma = \R^d\bs B_{L_\gamma}(0)$, and recall our definition of $L_n^q = \log n - C\logg n$. Then there are two cases to be checked.

{\underline{Case I:}} Suppose that $L_\gamma <L_n$. In this case, from Corollary \ref{cor:ad_con}, we have that $B_{r_\gamma}(\cP_\gamma \cap D_\gamma)$ has no connected subset of size $M_n$.
Therefore, we conclude that $\cP_\gamma \cap B_{L_\gamma}(0) \ne \emptyset$, and this implies that $B_{r_\gamma}(\cP_\gamma) \subset B_{L_\gamma + 2M_n r_\gamma}(0)$. Since $L_\gamma < L_n$,  we have
\[
r_\gamma^d = \frac{a(\log n)^{-1/\beta}e^{L_\gamma^q}}{nC_f} \le C_1 (\log n)^{-C-1/\beta},
\]
where $C_1 = a/C_f$, 
and
\[
2M_nr_\gamma\le C_2 \frac{(\log n)^{1-C/d-1/d\beta}}{\logg n},
\]
where $C_2 = 2\beta C_1^{1/d}$,
independently of $\gamma$. In particular, if $C>d$ we have that $2M_nr_\gamma \rightarrow 0$ uniformly over $\gamma$.
Next, we show that
\eqb\label{eqn:bound_diff}
(L_\gamma+2 M_n r_\gamma)^q - L_\gamma^q  \le \eps_n,
\eqe
for some $\eps_n\rightarrow 0$,
which implies that 
\eqb\label{eqn:dens_ratio}
\frac{f(L_\gamma)}{f(L_\gamma+2M_n r_\gamma)}\le 1+\eps'_n,
\eqe
for $\eps'_n\rightarrow 0$, where both $\eps_n,\eps_n'$ are independent of $\gamma$.

For $q\ge 1$, note that $L_\gamma \le (\log n)^{1/q}$, and therefore,
\[
(L_\gamma+2M_n r_\gamma)^q-L_\gamma^q \le 2q (L_\gamma+2M_nr_\gamma)^{q-1} M_n r_\gamma \le C_2 (\log n)^{(q-1)/q+1-C/d-1/d\beta} 
\le C_2(\log n)^{-{1/q}},
\]
if we assume that $C>2d$.
If $q \in (0,1)$, then we use the fact that $(a+b)^q \le a^q+b^q$, and then,
\[
(L_\gamma+2M_n r_\gamma)^q-L_\gamma^q \le (2 M_n r_\gamma)^q \le C_2^q \param{\frac{(\log n)^{1-C/d-1/d\beta}}{\logg n}}^q \le C_2 (\log n)^{-q}.
\]
Therefore, in \eqref{eqn:bound_diff} we can take $\eps_n = (\log n)^{-1/q}$ for $q\ge 1$ and $\eps_n = (\log n)^{-q}$ otherwise.

 Using Corollary \ref{cor:ad_cov}, we have 
\[
    B_{L_\gamma+2M_n r_\gamma}(0) \subset B_{R_n(L_\gamma+2M_n r_\gamma)}(\cP_n),
\]
implying that $\death(\gamma) \le R_n(L_\gamma + 2M_n r_\gamma)$. Therefore,
\[
    \pi(\gamma) \le \frac{R_n(L_\gamma+2 M_n r_\gamma)}{r_n(L_\gamma)} = \param{\frac{A\log n f(L_\gamma)}{a(\log n)^{-1/\beta}f(L_\gamma+2 M_n r_\gamma)}}^{1/d} \le C_3 (\log n)^{(1+1/\beta)/d}=  o(\pi_U),
\]
where we used \eqref{eqn:dens_ratio}, and this is a contradiction.

{\underline{Case II:}} Suppose $L_\gamma\ge L_n$. In this case $D_\gamma \subset D_n$. 
From Corollary \ref{cor:ad_small_comp}, we know that any component of $B_{r_\gamma}(\cP_\gamma\cap D_\gamma)$ consists of at most $m_0$ points (for some fixed $m_0>0$). Since $|\cP_\gamma| \ge M_n$, we conclude that $\cP_\gamma \cap B_{L_\gamma}(0) \ne \emptyset$, which again implies that
$B_{r_\gamma}(\cP_\gamma) \subset B_{L_\gamma+2m_0 r_\gamma}(0)$.
Note that from Corollary \ref{cor:ad_cov} we have 
\[
    B_{L_\gamma}(0) \subset B_{R_n(L_\gamma)}(\cP_n), 
\]
and therefore
\[
    B_{L_\gamma+2m_0r_\gamma}(0) \subset B_{R_n(L_\gamma)+2m_0r_\gamma}(\cP_n).
\]
This implies that $\death(\gamma) \le R_n(L_\gamma)+2m_0 r_\gamma$,
and therefore
\[
\pi(\gamma) \le \frac{R_n(L_\gamma)}{r_n(L_\gamma)} + 2m_0 \le \const (\log n)^{(1+1/\beta)/d} = o(\pi_U),
\]
for some constant $\const$, completing the proof.
\end{proof}

\section{Optimal Cycles -- Proofs}\label{sec:optimal}

As this section does not use any probability but we use a number of different tools, we  introduce the required notion and background to make this section self-contained. Throughout this section, we assume points are in general position. That is, no $d+1$ points lie on an affine hyperplane and no $d+2$ points lie on a $(d-1)$-sphere. Furthermore, in this section only, we assume that $B_r(x)$ represent open balls, while closed balls are denoted $\overline{B}_r(x)$. 

We also note that to ease exposition, and without loss of generality, we can fix the death time to be 1, and the center of the circumsphere corresponding to the \v Cech simplex which bounds the cycle is at the origin. Equivalently, this is the critical point of the distance function when the cycle is bounded.


 We next introduce the notion of the \emph{scale-normalized $k$-covering number} of a set $A\subset\R^d$, $\covnum_r(A)r^k,$ and remind the reader that this is known to converge to a limit as $r\rightarrow0$ and the limit only depends on the $k$-volume of the set Theorem~\ref{thm:covering}. \tocheck{This is also sometimes referred to as the density of the $k$-covering, but we refer to it simply as the \emph{normalized covering number} to avoid confusion with the densities arising in the geometric measure theory parts of the proof. }
 
 To relate it to persistence, we will use isoperimetric inequalities relating the filling radius and volume. 
The filling radius is defined as 
\[\FillRad (T)
=
\inf\left\{\rho>0: \exists W \;\text{s.t.}\; \spt(W)\subset B_\rho(T) \;\text{and}\; \partial W=T\right\}.\]
For a current $T$, we write $B_\rho(T):=B_\rho(\spt(T))$.

We refer the reader to notes by Guth~\cite{guthNotesGromovsSystolic2006} and the original result by Simon and Bombieri ~\cite{bombieriSimonGehring1983}. However, for our application, we will use a generalization by Gromov~\cite{gromovFillingRiemannianManifolds1983}.
\begin{thm}[\cite{gromovFillingRiemannianManifolds1983}, Theorem 8.1.C] \label{thm:sharp-fillrad-pseudomanifold}The filling radius of a closed $k$-dimensional 
 pseudomanifold  $A\subset \R^d$ admits the following (sharp) upper bound
 \[
\FillRad(A) \leq \left(\frac{\Vol_k(A)}{s_k}\right)^{1/k}\]
with equality for round $k$-spheres and is independent of the ambient dimension $d$.
\end{thm}
While the cited result is stated for submanifolds, the same argument applies verbatim to pseudomanifold parameterizations. For completeness, we provide the details in Appendix \ref{app:gromov_proof}, where we establish the inequality for arbitrary finite polyhedral cycles using Gromov’s visual-hull argument and deduce the corresponding statement for pseudomanifold parameterizations. Although the polyhedral formulation applies directly to geometric realizations of Čech cycles, we retain the pseudomanifold formulation to make explicit the geometric reduction underlying our application. The lemma below establishes this reduction in the main text; the appendix provides a detailed justification of the stated extension of Gromov’s theorem. We remind the reader that the sharp result is only known in Euclidean space.

To apply the result, we first show that a Čech cycle admits a pseudomanifold parameterization preserving its geometric $k$-volume. We remind the reader that all cycles have coefficients in $\mathbb{Z}_2$. Throughout this section, we use the notation $\geomreal{z}$ for the geometric realization of a chain $z$.

\begin{lem}
\label{lem:pseudomanifold-parametrization}
Let $P\subset \mathbb{R}^{d}$ be a locally finite set in affine
general position, and let $z$ be a  compactly supported simplicial $k$-cycle with $k<d$. Then there exists
 a finite, possibly disconnected, closed $k$-dimensional simplicial pseudomanifold $M$, and a simplex-wise affine map
$\varphi: M\rightarrow \mathbb{R}^{d},$
such that $\varphi_\#[M]$ is equal to $\geomreal{z}$, the geometric realization of $z$ (as a polyhedral cycle). Here $[M]$ denotes the fundamental cycle of $M$ and $\varphi_\#$ is the pushforward.
Additionally, the
pseudomanifold parameterization preserves the
$k$-dimensional measure, i.e.
\[
    \Vol_{k}(\varphi):= \int_M J_k \varphi\,d\mathcal{H}^{k}
    =
    \mathcal{H}^{k}\bigl(\varphi(M)\bigr).
\]
\end{lem}

\begin{proof}
Because $P$ is locally finite and $z$ has compact support, the
cycle $z$ contains only finitely many simplices. As we work over $\mathbb{Z}_2$, we may write
\[
    z=\sum_{\sigma\in\Sigma_z}\sigma,
\]
where $\Sigma_z$ is a finite collection of distinct
$k$-simplices and
 every $(k-1)$-simplex $\tau$ occurs an even number
of times in the boundaries of the simplices in $\Sigma_z$. Create an abstract copy of each $\sigma \in \Sigma_z$ denoted $
\tilde{\sigma}$. Note that here we are taking closed simplices  and so each copy includes the boundary of each simplex. Hence, for each $\tau$ there are an even number of abstract copies $\tilde{\tau}$. Pair the copies $\tilde{\tau}$ and quotient by the pairing. By construction, every abstract $(k-1)$-face  is incident
to exactly two \(k\)-simplices. This gives the required pseudomanifold which we denote $M$.
Define $\varphi\colon M\rightarrow\mathbb{R}^{d}$
on the copy \(\widetilde{\sigma}\), where
\(\sigma=[v_0,\ldots,v_k]\), by
\[
    \varphi\left(\sum_{i=0}^{k}t_i\widetilde v_i\right)
       :=
    \sum_{i=0}^{k}t_i v_i,
    \qquad
    t_i\geq 0,\qquad
    \sum_{i=0}^{k}t_i=1.
\]
The maps on paired faces agree, so they define a continuous simplex-wise affine map on $M$.
As the fundamental cycle of $M$ is the sum of all its
top-dimensional simplices
$    [M]=\sum_{\sigma\in\Sigma_z}
             \tilde{\sigma},
$
 we may write
\[
    \varphi_\#[M]
      =
    \sum_{\sigma\in\Sigma_z}\geomreal{\sigma}
      =
    \geomreal{z}.
\]
It remains to prove they have the same volume measure. Since $P$ is in
 general position and $k<d$, the vertices of every
$k$-simplex are affinely independent. Hence the restriction
\[
    \varphi\lvert _{\tilde{\sigma}}\colon
    \tilde{\sigma}\rightarrow \geomreal{\sigma} 
\]
is an affine embedding. By the area formula,
\[
    \int_{\widetilde{\sigma}}
        J_k\bigl(\varphi\rvert_{\widetilde{\sigma}}\bigr)\,
        d\mathcal{H}^{k}
      =
    \mathcal{H}^{k}(\geomreal{\sigma}).
\]
Summing over the top-dimensional simplices gives
\[
    \Vol _{k}(\varphi)
      =
    \sum_{\sigma\in\Sigma_z}
        \mathcal{H}^{k}(\geomreal{\sigma}).
\]
Finally, we show that the interiors of two distinct geometric
$k$-simplices overlap only on a set of $k$-dimensional measure
zero. Let \(\sigma,\sigma'\in\Sigma_z\) be distinct. As the vertices are in general position, $\aff(\sigma)$ and
$\aff(\sigma')$ are distinct $k$-dimensional
affine subspaces and their intersection has  dimension at most $k-1$.
Since \(\Sigma_z\) is finite, finite additivity up to null sets implies
$
    \Vol_{k}(\varphi)
       =
    \mathcal{H}^{k}\bigl(\varphi(M)\bigr)$, completing the proof. Note that here $\varphi(M)$ denotes the image set rather than the pushforward.
\end{proof}
Note that the result of the face pairing need not be simplicial but can be made so after subdivision. Additionally, the lemma does not assert that $\varphi(M)$ is itself an embedded
pseudomanifold. For example, four $k$-simplices may meet along the
same $(k-1)$-face. The abstract incidences can be paired to produce a pseudomanifold $M$, but different abstract faces of $M$ may have the same image. These identifications occur in dimension at
most \(k-1\), and therefore do not affect the $k$-dimensional
Hausdorff measure. The construction and conclusion
hold for every finite simplicial
cycle over $\mathbb{Z}_2$; the \v Cech construction is only used to define the geometric realization. Importantly, this allows us to apply Theorem~\ref{thm:sharp-fillrad-pseudomanifold} to sets which are parameterized by a pseudomanifold rather than only considering pseudomanifolds embedded in $\R^d$ and so applies to 
\v Cech cycles.

The result above applies to the filling radius of the geometric realization of a \v Cech cycle. This is not the same  as the \v Cech filtration, where the distances are taken from points rather than the entire realized cycle. Below we relate the two. 
 \begin{lem}\label{lem:close_1}
 Given a piecewise linear realization of a \v Cech cycle $\geomreal{z}$ with birth time $b$ and death time 1, for sufficiently small $b$,
\[\FillRad(\geomreal{z})\geq 1-o(1).\] 
 \end{lem}
 \begin{proof}
 Let $X$ denote the geometric realization of the set of vertices of $z$. 
We first show that for small enough $\delta$, there is an inclusion
 \[B_\delta(\geomreal{z})\subset B_{\sqrt{b^2+\delta^2}}(X)
 .\]
For any $q \in B_\delta(\geomreal{z})$, let $y$ denote the closest point in $\geomreal{z}$. As $y$ lies in the convex combination of a subset of $X$ and the distance of $y$ to the closest point in $X$ is less than $b$, the inclusion follows. 
If we assume  $\FillRad(\geomreal{z})<\sqrt{1-b^2}$, by the above inclusion this implies that the filling lies inside $\B_{t}(X)$
for $t<1$ which is a contradiction. Finally, 
 $\sqrt{1-b^2}>1-O(b^2)$, which is $1-o(1)$ as $b\rightarrow 0$, completing the proof.
 \end{proof}



Intuitively, we would like to use Gromov's sharp isoperimetric inequality connecting filling radius and volume along with the convergence of the normalized covering number to establish the optimality of the covering of a unit sphere for constructing cycles with large persistence with as few points as possible. A problem with this approach is that as $b\rightarrow 0$, the optimal cycle in terms of covering rather than volume may depend on $b$, which means we cannot use the convergence of the normalized covering number directly.

Hence, to prove the result, we will require tools from geometric measure theory. We begin by outlining the general strategy. We suppose that there exists a sequence of cycles, which can be generated by fewer points than a sphere, and our goal is to show that the existence of such a sequence implies a contradiction. A critical part of this proof is showing that the space of large persistence cycles are compact and so has a well-behaved limit. From this we can then construct the contradiction. Throughout this section, we consider numerous objects which are related but not the same. For example, we consider geometric realizations of \v Cech cycles and rectifiable currents, as well as coverings and the centers of the covers as measures. Throughout the section, we are as explicit as possible about the objects we are working with but this does require some notation. We denote the space of integral (rectifiable)\footnote{We remind the reader that all integral currents are rectifiable.} currents over $\mathbb{Z}_2$ in $\R^d$ by $\mathbf{I}_k(\mathbb{R}^d;\mathbb{Z}_2)$. Currents will always be taken over $\mathbb{Z}_2$, unless explicitly specified.  Additionally, we use $\mathbf{M}(T)$ and $\lVert T\rVert$ to denote the current mass and mass measure respectively.  Finally, we remind the reader that $s_k:=\Vol_k(\mathbb{S}^k)$ and $\coverconst$ is the
Euclidean $k$-dimensional normalized covering constant from Theorem~\ref{thm:covering}.

Most of the proof is devoted to establishing compactness. The first step is to show that every compact persistent cycle admits a finite volume bound at a small cost of the filling radius. This is sufficient to apply standard compactness results for currents in the flat topology. The next steps involve showing that the limit is well-behaved, namely that we can localize it to a compact set and that the filling radius is preserved in the limit. This is required as convergence in the flat-norm does not give sufficient control.

Fix the death time to 1, and consider a sequence of cycles $S_j$ with persistence $\pi_j$ with the birth time $b_j=1/\pi_j\rightarrow 0$, such that
$b_j^k\card{S_j}\leq (1-\delta)\coverconst\Vol^k(\mathbb S^k),$
for some fixed $\delta>0$ and where $\pi_j\geq \pi_0$ and $\pi_{0} \rightarrow \infty$. We remind the reader that by Lemma \ref{lem:close_1} $\FillRad(S_j)\geq 1-o(1)$. 

\begin{lem}
\label{lem:separated-net-regularization}
Let \(X\subset\mathbb R^d\) be finite, and let $\gamma$ 
be a persistent \v Cech class whose death time is normalized to \(1\) and assume the normalized birth time is $b$. 
Fix \(\eta>0\), set
$h=(1+\eta)b<1,$
and let \(Y\subset X\) be a maximal $(\eta b)$-separated subset. Then there exists a cycle $z_Y\in \CC_h(Y)$, with the geometric realization $S:=\geomreal{z_Y}$, such that $i(z_Y)$ is a representative of $\gamma$, where $i$ denotes the natural inclusion, with the following properties:
\begin{enumerate}
\item $
    \spt (S)
    \subseteq B_h(Y),$
\item $\death ([z_Y])\geq 1 $,
\item $\pi([z_Y])
    \geq\frac{1}{h}
    =\frac{\pi(\gamma)}{1+\eta},$
\item $\FillRad (S)\geq1-o(1).$
\item $\Vol(S)
    \leq C_{\mathrm{net}}(d,k,\eta)\,h^k\card{Y}.$
\end{enumerate}
\end{lem}
\begin{proof}
We first recall the standard construction of a maximal $(\eta b)$-separated set. Choose an arbitrary $x\in X$, and add it to $Y$. Iteratively add the point from $X\backslash Y$ that is furthest from $Y$ until all remaining points in $X$ are within $\eta b$ of a point in $Y$. 
The resulting $Y$ is an $(\eta b)$-separated maximal set. The sets are $\eta b$-interleaved by construction (see \cite{cohensteiner2007stability,chazal2009proximity} for more on interleaving). That is, for any $r>0$,
\[B_r(X) \subseteq B_{r+\eta b}(Y)\subseteq B_{r+\eta b}(X).\]
This follows from maximality. It follows that for $r=b$, a representative of $\gamma$ in $\CC_b(X)$ can be mapped to a representative in $\CC_h(Y)$. This immediately implies  (1); (2) follows from the fact that restricting to $Y$ can only increase death time; (3) follows immediately from (2) since birth time must be before $h$. By Lemma \ref{lem:close_1}, the filling radius on $X$ is at least $1-o(1)$, and $\eta b$-interleaving implies that the filling radius is at least $1-o(1)-\eta b$, and so $1-o(1)$. 
Finally, (5) follows from a standard packing argument for separated sets. As all points must be at least $\eta b$ separated, within a ball of radius $(1+\eta)b$, at most a constant number of other points can be packed and hence, 
each vertex is contained in a constant number of simplices. The bound immediately follows.
\end{proof}

\begin{rem}
\label{rem:regularization-order-of-limits}
In applications of Lemma~\ref{lem:separated-net-regularization},
\(\eta>0\) is kept fixed while the compactness limit is taken. This
ensures that \(C_{\mathrm{net}}(d,k,\eta)\) remains uniform. Only after obtaining
the limiting lower bound do we let \(\eta\rightarrow 0\).
\end{rem}
From this point on, we assume the sequence of cycles $S_j$ are $k$-currents. Polyhedral cycles such as the geometric realizations of \v Cech cycles naturally define integral currents (see
\cite[Sections~2.4 and~3.1--3.3]{MarcheseStuvard2018}. Additionally, under a general position assumption, the  volume of a geometric realization of a PL-cycle and the current mass are the same, i.e. $\Vol(\geomreal{z}) = \mathbf{M}(z)$.

We first use standard compactness and rectifiability results for flat
chains with $\mathbb Z_2$ coefficients; see the appendix
of \cite{White2009}.

By Lemma~\ref{lem:separated-net-regularization} (5), and the
assumed bound on the number of vertices, there exists
$M_0<\infty$ such that $\mathbf M(S_j)\le M_0$ for every $j$. As 
translations preserve mass and commute with the boundary operator, any translated cycles have uniformly bounded mass and since they are cycles, zero
boundary mass, so the following local flat compactness theorem applies. 
By \cite[Appendix, Theorem~A.1]{White2009}, applied with 
\(\mathbb R^d\) as the open set and coefficient group $\mathbb Z_2$, a
subsequence of the translated  $S_j$ converges locally in the flat
topology to a flat chain \(T\). Lower semicontinuity of mass gives
\(\mathbf{M}(T)\le M_0\), while continuity of the boundary gives
\(\partial T=0\). Finally,
\cite[Appendix, Theorem~A.4]{White2009} implies that \(T\) is rectifiable. 

%
%
The limit may, however, be zero. To obtain a nonzero limit
retaining large filling radius, we will first localize the
cycles. The following local lower mass bound is the first
step in this construction.
Define a concentration function
\[
\mathcal Q_j(R):=\sup\limits_{x\in \R^d} \mathbf{M}(S_j \llcorner B_R(x)),
\]
where $S_j \llcorner B_R(x)$ denotes the restriction of the current to $B_R(x)$.  

\begin{lem}\label{lem:noncollapsing}
For every $a>0$, if $\FillRad (S)\geq a$, then 
\begin{equation}\label{eq:localvolumebound}
\sup\limits_{x\in \R^d} (\mathbf{M}(S\llcorner B_{Ca}(x)))\geq ca^k
\end{equation}
\end{lem}
\begin{proof}
In this proof there are  constants $c$ and $C$ which can change meaning and value but depend only on $d$ and $k$. We use a localized Federer--Fleming deformation theorem. During the deformation, we must control both the support of the current swept out during the deformation and the mass of the deformed chain on each mesh face by the original mass in its star. These local estimates are stated explicitly in
\cite[Theorem~A.2]{BassoWengerYoung2023}.
Although formulated there for integer currents, the proof uses
radial projections and their associated homotopies, which work
equally with $\mathbb Z_2$ coefficients\footnote{For a deformation theorem for flat chains with coefficients
in an arbitrary complete normed abelian group, see
\cite[Theorem~1.1]{White1999Deformation}.}.

Consider a (hyper)-cubical mesh of size $ \ell=\lambda_{d,k}a$. We apply the simplicial deformation theorem to a fixed compatible
triangulation of the cubical mesh. Each resulting $k$-simplex has
volume bounded below by $c_{d,k}\ell^k$, and its star has diameter
at most $C_d\ell$. By \cite[Theorem~A.2]{BassoWengerYoung2023}, we can find  a deformation of $S$, denoted $S'$, which is a cellular cycle supported on the $k$-skeleton of the mesh.    $\lambda_{d,k}>0$ can be chosen small enough so that $S'$ is within $a/2$ of $\spt (S)$. Furthermore,  we can write $S = S'+\partial V$, where $V$ is a $(k+1)$-current with
\[
\spt (V)\subseteq  B_{C\ell}(\spt (S))\subseteq  B_ {a/2}(\spt (S)). 
\] 
Furthermore, for every $k$-face $\sigma$, $\mathbf{M}(S'\llcorner \sigma)\leq C\mathbf{M}(S\llcorner \meshstar (\sigma))$, where $\meshstar (\sigma)$ is the star of $\sigma$. 

Now assume the lemma does not hold. Since $\meshstar (\sigma)$ lies in a ball of radius $C_d\ell\leq a/2$, this implies      
\[\mathbf{M}(S'\llcorner \sigma)\leq c \cdot  C\cdot (a/2)^k,\]
as we have assumed  \eqref{eq:localvolumebound} does not hold. 
Choose $c$ small enough that 
\begin{equation}\label{eq:lessthan}
c\cdot  C\cdot a^k<c_{d,k}\ell^k,
\end{equation}
where $c_{d,k}\ell^k$ is the minimum volume of a simplex in the triangulation of the cubical mesh. 
However, since $S'$ is locally constant on the interior of $\sigma$ (and is over $\mathbb{Z}_2$), this implies that if $\mathbf{M}(S'\llcorner \sigma)>0$, then  
$\mathbf{M}(S'\llcorner \sigma)=\Vol_k(\sigma)\geq c_{d,k}\ell^k$,  contradicting \eqref{eq:lessthan}. Hence,  $S'$ must be 0 on all faces and so $\partial V=S$. Finally, since $\spt (V)\subset  B_{a/2}(\spt (S))$, this implies that $\FillRad (S)< a$ which is a contradiction. 
\end{proof}


While the above lemma gives a lower bound, it does not prevent mass and the filling-radius obstruction from splitting into widely separated pieces. We therefore localize before taking a limit. The next lemma shows that a bounded-mass cycle can be cut along a suitable sphere and closed by a filling having arbitrarily small mass and lying arbitrarily close to the cut. Iterating this construction will leave a remainder cycle with small filling radius; consequently, at least one of the bounded pieces must retain almost all of the filling radius of the original cycle.

\begin{lem}
\label{lem:localization}
Let \(1 \leq k < d\) and let \(M_0,\rho>0\). For every
\(0<\lambda_{\mathrm{loc}}<\rho\) and every \(m_{\mathrm{fill}}>0\), for every 
\(x\in\mathbb{R}^d\), there exists an $R_{\mathrm{loc}}=R_{\mathrm{loc}}(d,k,M_0,\rho,\lambda_{\mathrm{loc}},m_{\mathrm{fill}})<\infty$ such that for $S\in \mathbf{I}_k(\mathbb{R}^d;\mathbb{Z}_2)$, a compactly supported \(k\)-cycle satisfying $\mathbf{M}(S)\leq M_0,$  there is 
\begin{enumerate}
\item a $ t\in(2\rho,R_{\mathrm{loc}}-\lambda_{\mathrm{loc}})$ for which $S\llcorner B_t(x)\in \mathbf{I}_k(\mathbb R^d;\mathbb Z_2),$
and a \(k\)-current
$
    \mathcal C\in \mathbf{I}_k(\mathbb{R}^d;\mathbb{Z}_2)
$
such that
\begin{align*}
    \partial \mathcal C
    =
    \partial\bigl(S\llcorner B_t(x)\bigr),
    \qquad
    \mathbf{M}(\mathcal C)\leq \frac{m_{\mathrm{fill}}}{2},
\\
    \spt  (\mathcal C)
    \subset
    B_{\lambda_{\mathrm{loc}}}\!\!\left(
        \spt \left(
        \partial\bigl(S\llcorner B_t(x)\right)\right);
\end{align*}
\item
for $t$ and $\mathcal C$ as in (1), we can find  
\[
    T:=S\llcorner B_t(x)+\mathcal C,
    \qquad
    U:=\bigl(S-S\llcorner B_t(x)\bigr)+\mathcal C.
\]
such that $T$ and $U$ are integral \(k\)-cycles and 
\begin{align*}
    S=T+U&, \qquad
    \mathbf{M}(T)+\mathbf{M}(U)
   \leq
    \mathbf{M}(S)+m_{\mathrm{fill}}\\
    \spt(T)\subset B_{R_{\mathrm{loc}}}(x),
\quad &
    \spt(T)\cup\spt (U)
    \subset
    B_{\lambda_{\mathrm{loc}}}\!(\spt (S)), \quad
    T\llcorner B_\rho(x)
    =
    S\llcorner B_\rho(x).
\end{align*}
\end{enumerate}
\end{lem}



\begin{proof}
First suppose  that \(k\geq 2\). Let \(C_{\mathrm{iso}}=C_{\mathrm{iso}}(d,k)\)
be the constant in the support-controlled Euclidean isoperimetric
inequality, see \cite[p.~119, Federer--Fleming theorem]{guthNotesGromovsSystolic2006}. In $\R^d$, every \((k-1)\)-cycle \(Z\)  admits a
\(k\)-current $\mathcal C$ such that
$
\partial \mathcal C=Z,
$  with
\[
\mathbf{M}(\mathcal C)
 \leq C_{\mathrm{iso}}
       \mathbf{M}(Z)^{k/(k-1)},
\]
such that  $\spt (\mathcal C)
 \subset
 B _{C_{\mathrm{iso}}
 \mathbf{M}(Z)^{1/(k-1)}}
 (\spt(Z)).$

For $k\geq 2$, choose  
\(0<\delta_{\mathrm{slice}}<1\) sufficiently small so that
$
C_{\mathrm{iso}}\delta_{\mathrm{slice}}^{1/(k-1)}
 \leq \min\{\frac{m_{\mathrm{fill}}}{2},\lambda_{\mathrm{loc}}\}
$;
   for \(k=1\) any \(\delta_{\mathrm{slice}}\in(0,1)\) is sufficient.
%
%
%
Let \(J_{\mathrm{loc}}\in\mathbb N\) be sufficiently large so that
$\frac{M_0}{J_{\mathrm{loc}}\rho}<\delta_{\mathrm{slice}},$
and set $
R_{\mathrm{loc}}:=(J_{\mathrm{loc}}+2)\rho+\lambda_{\mathrm{loc}}.$
If we denote
$
u_x(y):=\dist(x,y) 
$, i.e. the distance to $x$,
by the slicing inequality,
\begin{align*}
\int_{2\rho}^{(J_{\mathrm{loc}}+2)\rho}
\mathbf{M}
\bigl(\langle S,u_x,t\rangle\bigr)\,dt
&\leq
\mathbf{M}\left(
S\llcorner
\bigl(B_{(J_{\mathrm{loc}}+2)\rho}(x)\setminus B_{2\rho}(x)\bigr)
\right) \\
&\leq \mathbf{M}(S) \leq M_0,
\end{align*}
where $\langle S,u_x,t\rangle$ is the slice of the $k$-current $S$ by $u_x^{-1}(t)=\partial B_t(x)$.  
Since the interval has length \(J_{\mathrm{loc}}\rho\), there exists
$
t\in\bigl(2\rho,(J_{\mathrm{loc}}+2)\rho\bigr)$
such that \(S\llcorner B_t(x)\) is integral and 
\[
\mathbf{M}
\bigl(\langle S,u_x,t\rangle\bigr)
\leq\frac{M_0}{J_{\mathrm{loc}}\rho}
<\delta_{\mathrm{slice}}.
\]
By rectifiability, for almost every $t$,
\(\langle S,u_x,t\rangle
=
\partial\bigl(S\llcorner\{u_x<t\}\bigr)
-
(\partial S)\llcorner\{u_x<t\}.\) Since 
$\partial S=0$, it follows that $\langle S,u_x,t\rangle = \partial(S\llcorner B_t(x))$ and therefore $\mathbf{M}(\partial(S\llcorner B_t(x)))<\delta_{\mathrm{slice}}$. Set $Z:=\partial(S\llcorner B_t(x))$. 
For \(k\geq2\), we apply the support-controlled isoperimetric inequality referenced above to $Z$, which provides 
a current \(\mathcal C\in\mathbf I_k(\mathbb R^d;\mathbb Z_2)\), where $\partial \mathcal C=Z$  and 
\begin{align*}
\mathbf{M}(\mathcal C)
&\leq
C_{\mathrm{iso}}
\mathbf{M}(Z)^{k/(k-1)}
\leq\frac{m_{\mathrm{fill}}}{2},
\\
\spt(\mathcal C)
&\subset
B_{C_{\mathrm{iso}}
\mathbf{M}(Z)^{1/(k-1)}}
(\spt (Z))
\subset
B_{\lambda_{\mathrm{loc}}}\!(\spt (Z)),
\end{align*}
where the last inclusion follows by applying  support control to the mass bound on $Z$.

For \(k=1\), then \(Z\) is a zero-dimensional integral current. Hence its mass is a nonnegative integer, and
$
\mathbf{M}(Z)<\delta_{\mathrm{slice}}<1.
$
It follows that \(Z=0\), and we can take \(\mathcal C=0\). 
Hence, we have found the required $t$ and $\mathcal C$
for all \(k\geq1\) and the proof of assertion (1) is complete.


To prove assertion (2), we first define
\[
    S^{\mathrm{in}}:=S\llcorner B_t(x),
    \qquad
    S^{\mathrm{out}}:=S-S^{\mathrm{in}}.
\]
This separates $S$ into the restriction to the ball and everything else. By definition,
\[
    \partial S^{\mathrm{in}}
    =\partial\bigl(S\llcorner B_t(x)\bigr)
    =\partial \mathcal C.
\]
Since $\partial S=0$ and we are working over $\mathbb{Z}_2$, we also have
$
    \partial S^{\mathrm{out}}
    =\partial S^{\mathrm{in}}.
$
Next we define
\[
    T:=S^{\mathrm{in}}+\mathcal C,
    \qquad
    U:=S^{\mathrm{out}}+\mathcal C.
\]
Both are integral $k$-currents, and
\[
    \partial T
    =\partial S^{\mathrm{in}}+\partial \mathcal C=0,
    \qquad
    \partial U
    =\partial S^{\mathrm{out}}+\partial \mathcal C=0,
\]
so $T$ and $U$ are integral $k$-cycles. Over
$\mathbb Z_2$, we therefore have
\[
    T+U
    =S^{\mathrm{in}}+S^{\mathrm{out}}+\mathcal C+\mathcal C
    =S, 
\]
which is the first condition in assertion (2).
For almost all $t$, $S\llcorner \partial B_t(x)$ has no $k$-volume, so the inside and outside restrictions split the mass of
$S$
\[
    \mathbf{M}(S^{\mathrm{in}})+\mathbf{M}(S^{\mathrm{out}})
    =\mathbf{M}(S).
\]
Using the triangle inequality and the estimate
$\mathbf{M}(\mathcal C)\leq m_{\mathrm{fill}}/2$ from part~\textnormal{(1)}, we obtain
\begin{align*}
    \mathbf{M}(T)+\mathbf{M}(U)
    &\leq
    \mathbf{M}(S^{\mathrm{in}})+\mathbf{M}(S^{\mathrm{out}})
    +2\mathbf{M}(\mathcal C) \\
    &\leq \mathbf{M}(S)+m_{\mathrm{fill}},
\end{align*}
as required. 
It remains to verify the support properties. Since  
$
    \spt(\partial\left(S\llcorner B_t(x)\right))
    \subset \spt (S)\cap\partial B_t(x),$
part (1) gives
\[
    \spt(\mathcal C)
    \subset
    B_{\lambda_{\mathrm{loc}}}\!\left(
        \spt\partial\bigl(S\llcorner B_t(x)\bigr)
    \right)
    \subset B_{t+\lambda_{\mathrm{loc}}}(x).
\]
As $t<R_{\mathrm{loc}}-\lambda_{\mathrm{loc}}$, it follows that $\spt (\mathcal C) \subset B_{R_{\mathrm{loc}}}(x)$. Since
$\spt (S^{\mathrm{in}})\subset \overline{B}_t(x)\subset B_{R_{\mathrm{loc}}}(x)$ and $T=S^{\mathrm{in}}+\mathcal C$, it follows that 
    $\spt(T)\subset B_{R_{\mathrm{loc}}}(x),$ proving the required condition. 

Next, we combine the facts that
$\spt (S^{\mathrm{in}})\cup\spt (S^{\mathrm{out}})
    \subset\spt(S)$
and
$   \spt (\mathcal C)
    \subset
    B_{\lambda_{\mathrm{loc}}}\!\left(
        \spt(\partial\bigl(S\llcorner B_t(x)\bigr))
    \right)
    \subset B_{\lambda_{\mathrm{loc}}}(\spt (S))
$ from (1),
to obtain
\[
    \spt(T)\cup\spt(U)
    \subset B_{\lambda_{\mathrm{loc}}}\!(\spt (S)).
\]
We remind the reader that $t>2\rho$ and $\lambda_{\mathrm{loc}}<\rho$. Since the support of
$\partial(S\llcorner B_t(x))$ lies on $\partial B_t(x)$, 
\[
    \spt(\mathcal C)
    \subset\mathbb R^d\setminus B_{t-\lambda_{\mathrm{loc}}}(x)
    \subset\mathbb R^d\setminus B_\rho(x).
\]
As $\spt(\mathcal C)$ and $B_\rho(x)$ are disjoint, it therefore follows that $\mathcal C\llcorner B_\rho(x)=0$. As $B_\rho(x)\subset B_t(x)$, $S^{\mathrm{out}}\llcorner B_\rho(x) =0 $, so we can conclude
that
\[
    T\llcorner B_\rho(x)
    =S^{\mathrm{in}}\llcorner B_\rho(x)
    =S\llcorner B_\rho(x),
\]
completing the proof.
\end{proof}
We now iterate the above construction to show there is a localized cycle with large filling radius.
    
\begin{lem}
\label{lem:persistent-localized-cycle}
Let
$
S_j\in\mathbf I_k(\mathbb R^d;\mathbb Z_2)
$
be compactly supported  \(k\)-cycles such that
\begin{align*}
\partial S_j=0,
\qquad
\mathbf{M}(S_j)\leq M_0, \\
\FillRad (S_j)\geq 1-\xi_j,
\qquad
\xi_j\rightarrow0.
\end{align*}
For every \(\varepsilon\in(0,1)\) and \(\zeta>0\),
there exists an
$R_{\mathrm{loc}}<\infty$
and \(j_0\in\mathbb N\) such that, for every \(j\geq j_0\), one can
choose an integral \(k\)-cycle
$
T_j\in\mathbf I_k(\mathbb R^d;\mathbb Z_2)
$
and a point \(x_j\in\mathbb R^d\) satisfying
\[
\partial T_j=0,
\qquad
\spt(T_j)\subset B_{R_{\mathrm{loc}}}(x_j),
\]
\[
\spt T_j
\subset B_\zeta(\spt S_j),
\qquad
\mathbf{M}(T_j)\leq M_0+\zeta,\qquad
\FillRad (T_j)\geq 1-\varepsilon.
\]
\end{lem}

\begin{proof}
Fix \(\varepsilon\in(0,1)\) and \(\zeta>0\), and set
$\rho:=\frac{\varepsilon}{8}.$
By Lemma \ref{lem:noncollapsing}, there  exists an
$
m_\rho:=c_{d,k}\rho^k>0$, 
such that every compactly supported integral \(k\)-cycle \(U\) with $
\FillRad (U)>\rho$
admits a point \(x\in\mathbb R^d\) satisfying
$
\mathbf{M}
\bigl(U\llcorner B_\rho(x)\bigr)
\geq m_\rho.
$
Choose an integer
\[
q_*>
\frac{2M_0}{m_\rho},
\]
then choose \(0<\lambda_{\mathrm{loc}}<\rho\) and \(m_{\mathrm{fill}}>0\) sufficiently small that
\[
q_*\lambda_{\mathrm{loc}}
<
\min\left\{\zeta,\frac{\varepsilon}{8}\right\},\qquad
m_{\mathrm{fill}}<\frac{m_\rho}{2},
\qquad
q_*m_{\mathrm{fill}}<\zeta,\qquad
A_{d,k}\left(\frac{m_{\mathrm{fill}}}{2}\right)^{1/k}<\lambda_{\mathrm{loc}};
\]
where \(A_{d,k}>0\) is a dimensional constant in the support
estimate for mass-minimizing fillings obtained from monotonicity below \eqref{eq:mass_min}.

Apply Lemma~\ref{lem:localization} with mass bound \(M_0+\zeta\),
localization scale \(\rho\), support error \(\lambda_{\mathrm{loc}}\), and mass
error \(m_{\mathrm{fill}}\). Denote the resulting localization radius by
$
R_{\mathrm{loc}}$.
For a fixed \(j\), set
$
U_j^0:=S_j$.
We inductively construct bounded localized cycles \(T_j^\ell\) and remainder
cycles \(U_j^\ell\). The stopping criterion is that the remainder cycles have a small filling radius, i.e.
$
\FillRad (U_j^{\ell-1})\leq\rho.$
Otherwise, Lemma \ref{lem:noncollapsing} provides a point 
\(x_j^\ell\in\mathbb R^d\) such that
\[
\mathbf{M}
\bigl(
U_j^{\ell-1}\llcorner B_\rho(x_j^\ell)
\bigr)
\geq m_\rho,
\]
Lemma~\ref{lem:localization}, applied to \(U_j^{\ell-1}\) at this
\(x_j^\ell\), gives a cutting radius
\[
t_j^\ell\in(2\rho,R_{\mathrm{loc}}-\lambda_{\mathrm{loc}})
\]
and  \(\widetilde{\mathcal C}_j^\ell\), a filling of $Z_j^\ell$   where
\[
Z_j^\ell
:=\partial\bigl(U_j^{\ell-1}\llcorner B_{t_j^\ell}(x_j^\ell)\bigr),
\]
with 
\[
\partial\widetilde{\mathcal C}_j^\ell=Z_j^\ell,
\qquad
\mathbf{M}(\widetilde{\mathcal C}_j^\ell)\leq\frac{m_{\mathrm{fill}}}{2}.
\]
Before defining the localized cycle and the next remainder, we replace the
filling by a mass-minimizing filling with the same boundary.
If $Z_j^\ell=0$, take $\mathcal C_j^\ell=0$; all the required
mass and support bounds then hold trivially. Otherwise, consider
a minimizing sequence of compactly supported integral $k$-currents
with boundary $Z_j^\ell$.
The candidate $\widetilde{\mathcal C}_j^\ell$ bounds the infimum.
Nearest-point projection onto the compact convex hull of
$\spt(Z_j^\ell)$ confines the minimizing sequence to a fixed
compact set without increasing mass or changing its boundary.
Compactness and rectifiability
\cite[Appendix, Lemma~A.2 and Theorem~A.4]{White2009},
together with lower semicontinuity of mass, therefore give a
minimizer $\mathcal C_j^\ell$ satisfying
\[
\partial\mathcal C_j^\ell=Z_j^\ell,
\qquad
\mathbf M(\mathcal C_j^\ell)
\leq \mathbf M(\widetilde{\mathcal C}_j^\ell)
\leq \frac{m_{\mathrm{fill}}}{2}.
\]
%
%
 Since $\mathcal C_j^\ell$ is mass minimizing, away from $\spt(Z_j^\ell)$, the monotonicity formula
\cite[Chapter~4, Section~3, (3.6)--(3.10)]{SimonGMT}
 implies that
$\frac{\mathbf M(\mathcal C_j^\ell\llcorner B_r(y))}
     {\omega_k r^k}$
is nondecreasing for
$0<r<\dist(y,\spt(Z_j^\ell))$, where $\omega_k$ is the volume
of the unit $k$-ball. Since $\mathcal C_j^\ell$ must have density 1 almost everywhere, for every $y\in\spt(\mathcal C_j^\ell)$
\begin{equation}\label{eq:mass_min}
\dist(y,\spt(Z_j^\ell))
\leq A_{d,k}\mathbf M(\mathcal C_j^\ell)^{1/k}.
\end{equation}
The mass bound and the choice of $m_{\mathrm{fill}}$ imply
\[
A_{d,k}\mathbf M(\mathcal C_j^\ell)^{1/k}
\leq A_{d,k}\left(\frac{m_{\mathrm{fill}}}{2}\right)^{1/k}
<\lambda_{\mathrm{loc}}.
\]
Therefore, 
$
\spt(\mathcal C_j^\ell)
\subset
B_{\lambda_{\mathrm{loc}}}\!(\spt(Z_j^\ell)).
$
Finally, the slice satisfies
$
\spt(Z_j^\ell)
\subset
\spt(U_j^{\ell-1})
\cap\partial B_{t_j^\ell}(x_j^\ell),
$
 so the filling lies within $\lambda_{\mathrm{loc}}$ of the
support of the previous remainder and of the cutting sphere. 
Since
$
t_j^\ell-\lambda_{\mathrm{loc}}>\rho$ (by \(0<\lambda_{\mathrm{loc}}<\rho\))  and 
$
t_j^\ell+\lambda_{\mathrm{loc}}<R_{\mathrm{loc}}$,
we conclude that
\[
\spt(\mathcal C_j^\ell)
\subset
B_{R_{\mathrm{loc}}}(x_j^\ell)\setminus B_\rho(x_j^\ell),
\qquad
\spt(\mathcal C_j^\ell)
\subset
B_{\lambda_{\mathrm{loc}}}(\spt(U_j^{\ell-1})).
\]
%
Using this same minimizing filling on both sides of the cut, define
\begin{align*}
T_j^\ell
&:=U_j^{\ell-1}\llcorner B_{t_j^\ell}(x_j^\ell)+\mathcal C_j^\ell,\\
U_j^\ell
&:=\bigl(U_j^{\ell-1}
-U_j^{\ell-1}\llcorner B_{t_j^\ell}(x_j^\ell)\bigr)+\mathcal C_j^\ell.
\end{align*}
These are integral \(k\)-cycles, and the cancellation of  \(\mathcal C_j^\ell\) over \(\mathbb Z_2\) gives
$U_j^{\ell-1}=T_j^\ell+U_j^\ell.$
Moreover,
\begin{align*}
\mathbf{M}(T_j^\ell)
+\mathbf{M}(U_j^\ell)
&\leq\mathbf{M}(U_j^{\ell-1})
+2\mathbf{M}(\mathcal C_j^\ell)\\
&\leq\mathbf{M}(U_j^{\ell-1})+m_{\mathrm{fill}},
\end{align*}
along with 
\[
\spt (T_j^\ell)\subset B_{R_{\mathrm{loc}}}(x_j^\ell),
\qquad
\spt (T_j^\ell)\cup\spt (U_j^\ell)
\subset
B_{\lambda_{\mathrm{loc}}}\!(\spt (U_j^{\ell-1})),
\]
\[
T_j^\ell\llcorner B_\rho(x_j^\ell)
=U_j^{\ell-1}\llcorner B_\rho(x_j^\ell).
\]
The next cut is applied to this remainder \(U_j^\ell\).
In particular,
$
\mathbf{M}(T_j^\ell)\geq m_\rho,
$
so that 
\begin{align*}
\mathbf{M}(U_j^\ell)
&\leq
\mathbf{M}(U_j^{\ell-1})
-\mathbf{M}(T_j^\ell)+m_{\mathrm{fill}} \\
&\leq
\mathbf{M}(U_j^{\ell-1})
-m_\rho+m_{\mathrm{fill}} \\
&\leq
\mathbf{M}(U_j^{\ell-1})
-\frac{m_\rho}{2}.
\end{align*}
After \(q\) such cuts, 
\[
\mathbf{M}(U_j^q)
\leq
M_0-\frac{q\,m_\rho}{2}.
\]
As we have chosen $q_*>\frac{2M_0}{m_\rho}$, we conclude that the construction must stop after at most 
$q_j < q_*$ cuts. 
At the stopping time,
$
\FillRad (U_j^{q_j})\leq\rho$
and summing all the previous steps gives
\[
S_j
=
\sum_{\ell=1}^{q_j}T_j^\ell
+
U_j^{q_j}.
\]
As all localization fillings \(\mathcal C_j^1,\ldots,\mathcal C_j^{q_j}\) used in this
construction are mass minimizing with their respective boundaries
fixed, we have 
\[
\sum_{\ell=1}^{q_j}\mathbf{M}(\mathcal C_j^\ell)
\leq\frac{q_jm_{\mathrm{fill}}}{2}<\frac{\zeta}{2}.
\]
Summing the mass estimates over all cuts also yields
\[
\sum_{\ell=1}^{q_j}
\mathbf{M}(T_j^\ell)
+
\mathbf{M}(U_j^{q_j})
\leq
\mathbf{M}(S_j)+q_jm_{\mathrm{fill}}
<
M_0+\zeta.
\]
In particular,
$
\mathbf{M}(T_j^\ell)<M_0+\zeta$
for every \(\ell\).
Iterating the support estimates
\[
\spt (U_j^\ell)
\subset
B_{\ell\lambda_{\mathrm{loc}}}\!(\spt (S_j)).
\]
Combining this with $\spt (T_j^\ell)\cup\spt (U_j^\ell)
\subset
B_{\lambda_{\mathrm{loc}}}\!(\spt (U_j^{\ell-1}))$ gives
\[
\spt (T_j^\ell)
\subset
B_{\ell\lambda_{\mathrm{loc}}}\!(\spt (S_j)).
\]
Since \(\ell\leq q_j<q_*\), it follows that
\[
\spt (T_j^\ell)
\cup
\spt (U_j^{q_j})
\subset
B_{q_*\lambda_{\mathrm{loc}}}\!(\spt (S_j))
\subset
B_\zeta(\spt (S_j)).
\]
Hence, after fewer than \(q_*\) localization steps, we have decomposed
\[
S_j=\sum_{\ell=1}^{q_j}T_j^\ell+U_j^{q_j},
\qquad q_j<q_*,
\]
where each localized cycle \(T_j^\ell\) is supported in a ball \(B_{R_{\mathrm{loc}}}(x_j^\ell)\), all localized cycles and the terminal remainder remain in \(B_\zeta(\spt(S_j))\) and
$\FillRad (U_j^{q_j})\le\rho.$

It remains to show that at least one of the localized cycles \(T_j^\ell\) retains filling radius at least \(1-\varepsilon\). Suppose that
$
\FillRad (T_j^\ell)<1-\varepsilon
$ 
for every \(\ell=1,\ldots,q_j\). For each \(\ell\),  we can choose an integral
\((k+1)\)-current \(V_j^\ell\) such that
\[
\partial V_j^\ell=T_j^\ell, \qquad
\spt V_j^\ell
\subset
B_{1-\varepsilon}(\spt (T_j^\ell)).
\]
Likewise, since 
$
\FillRad (U_j^{q_j})\leq\rho,
$
there also exists an integral \((k+1)\)-current \(W_j\) such that
\[
\partial W_j=U_j^{q_j}, \qquad \spt (W_j)
\subset
B_{\rho+\varepsilon/8}
(\spt (U_j^{q_j})).
\]
We now define
$
V_j:=\sum_{\ell=1}^{q_j}V_j^\ell+W_j
$
and observe that 
$
\partial V_j
=
\sum_{\ell=1}^{q_j}T_j^\ell+U_j^{q_j}
=
S_j.
$
Additionally,
\[
\spt (V_j^\ell)
\subset
B_{1-\varepsilon+q_*\lambda_{\mathrm{loc}}}\!
(\spt (S_j)),\quad
\spt (W_j)
\subset
B_{\rho+\varepsilon/8+q_*\lambda_{\mathrm{loc}}}\!
(\spt (S_j)).
\]
Writing $
\rho_\varepsilon
:=
\max\left\{
1-\varepsilon+q_*\lambda_{\mathrm{loc}},\,
\rho+\frac{\varepsilon}{8}+q_*\lambda_{\mathrm{loc}}
\right\},
$
we obtain $
\spt (V_j)
\subset
B_{\rho_\varepsilon}(\spt (S_j)).
$

Since $
q_*\lambda_{\mathrm{loc}}<\frac{\varepsilon}{8}$ and $
\rho=\frac{\varepsilon}{8},$
we have
$
\rho_\varepsilon
<
\max\left\{
1-\frac{7\varepsilon}{8},\,
\frac{3\varepsilon}{8}
\right\}
<1.
$
Furthermore, since $
\spt (V_j)
\subset
B_{\rho_\varepsilon}(\spt S_j)
$, we have 
$\FillRad (S_j)\leq\rho_\varepsilon.
$
As \(\xi_j\rightarrow0\), for all large enough $j$, 
$
1-\xi_j>\rho_\varepsilon $. Hence, $\FillRad (S_j)\leq\rho_\varepsilon
$ contradicts the  assumption
$\FillRad (S_j)\geq1-\xi_j.$

 We conclude that for every \(j\geq j_0\), we can find an $\ell_j$ such that
$
\FillRad (T_j^{\ell_j})
\geq1-\varepsilon.
$
Setting
\[
T_j:=T_j^{\ell_j},
\qquad
x_j:=x_j^{\ell_j},
\]
we have $\FillRad (T_j)\geq1-\varepsilon$ and 
by the estimates above, we have
\begin{align*}
\partial T_j=0,
\qquad&
\spt (T_j)\subset B_{R_{\mathrm{loc}}}(x_j),\\
\spt (T_j)
\subset B_\zeta(\spt (S_j)),
\qquad&
\mathbf{M}(T_j)\leq M_0+\zeta.
\end{align*}
\end{proof}


We recap where we are in the proof. Since translations preserve boundary, mass, and filling radius,
fixing \(\varepsilon\in(0,1)\) and \(\zeta>0\), and letting \(T_j\) and
\(x_j\) be those given by Lemma~\ref{lem:persistent-localized-cycle},  we can translate $T_j$ by $-x_j$ denoting the result by $\widehat{T}_j$, such that 
\[
\partial \widehat T_j=0,
\qquad
\spt ( \widehat T_j)\subset\overline{B}_{R_{\mathrm{loc}}}(0),
\qquad
\mathbf{M}(\widehat T_j)\leq M_0+\zeta.
\]
The Federer--Fleming compactness theorem  gives, after
passing to a subsequence, a current
$
T_{\varepsilon,\zeta}
\in\mathbf I_k(\mathbb R^d;\mathbb Z_2)
$
such that
$
\widehat T_j\rightarrow T_{\varepsilon,\zeta}$
where convergence is in the global flat norm.
Continuity of the boundary operator gives $
\partial T_{\varepsilon,\zeta}=0
 $ and lower semicontinuity of mass give $
\mathbf{M}(T_{\varepsilon,\zeta})
\leq
\liminf_{j\rightarrow\infty}\mathbf{M}(\widehat T_j)
\leq M_0+\zeta.
$   
Moreover, the support is controlled, 
$\spt (T_{\varepsilon,\zeta})
\subset\overline{B}_{R_{\mathrm{loc}}}(0).
$
%

We now show that the filling radius is preserved in the limit. 
\begin{lem}
\label{lem:fillrad-semicontinuity}
Let
$
A_j,A\in\mathbf I_k(\mathbb R^d;\mathbb Z_2)
$
be compactly supported integral \(k\)-cycles such that
$
A_j\rightarrow A$
in the global flat norm.
Then
\[
\FillRad (A)
\geq
\limsup_{j\rightarrow\infty}\FillRad (A_j).
\]
\end{lem}

\begin{proof}
Set
$D_j:=A_j-A.$
Since \(A_j\rightarrow A\) in the global flat norm, there exist
$
U_j\in\mathbf I_k(\mathbb R^d;\mathbb Z_2)
$ and  $
F_j\in\mathbf I_{k+1}(\mathbb R^d;\mathbb Z_2),$
such that $
D_j=U_j+\partial F_j$
with
\[
\mathbf{M}(U_j)
+
\mathbf{M}(F_j)
\leq
\flatnorm{D_j}+\frac1j
\rightarrow0,
\]
where $\flatnorm{\cdot}$ denotes the global flat norm. 
Because \(D_j\) is a cycle,
$
\partial U_j=0.
$
By the Euclidean isoperimetric inequality, there exists
$
C_j\in\mathbf I_{k+1}(\mathbb R^d;\mathbb Z_2)
$
such that
$
\partial C_j=U_j
$  with 
\[
\mathbf{M}(C_j)
\leq
C_{d,k}\mathbf{M}(U_j)^{(k+1)/k}.
\]
Defining 
$
E_j:=F_j+C_j
$, we observe that 
$
\partial E_j=D_j
$
and so 
\[
\mathbf{M}(E_j)
\leq
\mathbf{M}(F_j)
+
C_{d,k}\mathbf{M}(U_j)^{(k+1)/k}
\rightarrow0.
\]
Now let \(Q_j\) be a mass-minimizing integral \((k+1)\)-current
with
$
\partial Q_j=D_j.
$
Then by definition, 
$
\mathbf{M}(Q_j)
\leq
\mathbf{M}(E_j)
\rightarrow0.
$
Next, we claim that
\[
\spt (Q_j)
\subset
B_{a_j}(\spt (D_j))
\]
for some \(a_j\rightarrow0\). Let
$
y\in\spt (Q_j)
$
and set
$
r_y:=
\dist
\bigl(y,\spt (D_j)\bigr).$ 
If \(r_y>0\), then \(Q_j\) is locally mass minimizing in \(B_{r_y}(y)\). For every \(0<s<r_y\), 
the monotonicity formula (as in Lemma ~\ref{lem:persistent-localized-cycle}, see \cite[Section~12, Equation~(12.2)]{AmbrosioKatz2011}) and the positive density of an integral
current give
\[
\mathbf{M}
\bigl(Q_j\llcorner B_s(y)\bigr)
\geq
c_{d,k}s^{k+1}
\]
Taking \(s=r_y/2\), 
\begin{align*}
\mathbf{M}(Q_j)
\geq
c_{d,k}\left(\frac{r_y}{2}\right)^{k+1}
\quad \Rightarrow \quad
r_y
\leq
2c_{d,k}^{-1/(k+1)}
\mathbf{M}(Q_j)^{1/(k+1)}.
\end{align*}
Choosing 
$
a_j>
2c_{d,k}^{-1/(k+1)}
\mathbf{M}(Q_j)^{1/(k+1)}
$ with $a_j \rightarrow 0$, we obtain
\[
\spt (Q_j)
\subset
B_{a_j}(\spt (D_j)).
\]
We next show the one-sided convergence of supports:
\[
e_j:=
\sup_{x\in\spt (A)}
\dist (x,\spt (A_j))
\rightarrow0.
\]
Suppose that this is not the case. After passing to a subsequence, there would exist
\(\theta>0\) and points \(x_j\in\spt (A)\) such that
$
B_\theta(x_j)\cap\spt (A_j)=\emptyset.
$
Since \(\spt (A)\) is compact, we can pass to a further subsequence such that
$
x_j\rightarrow x\in\spt (A),
$
so for all sufficiently large \(j\),
$
A_j\llcorner B_{\theta/2}(x)=0.
$
Since flat convergence implies weak convergence, it follows that 
$
A\llcorner B_{\theta/2}(x)=0,
$
contradicting \(x\in\spt (A)\) so 
\(e_j\rightarrow0\).

Now assume  that \(A\neq0\), and let
$
r>\FillRad (A)
$. 
By the definition of filling radius, there exists a
$
V\in\mathbf I_{k+1}(\mathbb R^d;\mathbb Z_2)
$
such that
$
\partial V=A
$
and
$
\spt (V) \subset
B_r(\spt (A))$.
Since
$\spt (A)
\subset
\overline{B}_{e_j}(\spt (A_j))
$,
we have
$
\spt (V)
\subset
B   _{r+e_j}(\spt (A_j)).
$

As we can write
$
\spt (D_j)
\subset
\spt (A_j)\cup\spt (A),
$
it follows that 
$
\spt (Q_j)
\subset
B_{a_j+e_j}(\spt (A_j)).
$
The current
$
V_j:=V+Q_j
$
satisfies
$
\partial V_j
=
A+D_j
=
A_j
$
and its support is contained in
$
B_{\rho_j}(\spt (A_j)),
$
where
$
\rho_j
:=
\max\left\{
r+e_j,\,
a_j+e_j
\right\}.
$
This implies that
$
\FillRad (A_j)
\leq
\rho_j.
$
Since \(a_j,e_j\rightarrow0\),
\[
\limsup_{j\rightarrow\infty}\FillRad (A_j)
\leq r.
\]
Letting \(r\rightarrow \FillRad (A)\) from above gives
\[
\limsup_{j\rightarrow\infty}\FillRad (A_j)
\leq
\FillRad (A).
\]
Finally, if \(A=0\), then \(D_j=A_j\), and \(Q_j\) itself is a filling
of \(A_j\) satisfying
$
\spt(Q_j)
\subset
B_{a_j}(\spt (A_j)).
$
Hence
\[
\FillRad (A_j)
\leq a_j\rightarrow0
=
\FillRad (A).
\]
completing the proof.
\end{proof}
Applying Lemma~\ref{lem:fillrad-semicontinuity} to the convergence obtained above, i.e., 
$ \widehat T_j\rightarrow T_{\varepsilon,\zeta},$ gives
\[
\FillRad (T_{\varepsilon,\zeta})
\geq
\limsup_{j\rightarrow\infty}
\FillRad (\widehat T_j)
\geq
1-\varepsilon.
\]
This ensures that the localized cycles retain large filling radius
in the limit. However, the localization procedure introduces fillings that may increase the covering number.
 We next show that, after slightly enlarging the
covering radius, the covering centers provided  by Lemma~\ref{lem:separated-net-regularization} may be augmented in such a way that the additional centers needed to cover these fillings can be controlled by their total mass.
\begin{lem}
\label{claim:cover-localization-fillings}
Suppose that $Y_j\subset \R^d$ are finite sets covering $S_j$, i.e. 
\[
\spt(S_j)
\subset
B_{h_j}(Y_j),
\qquad h_j\rightarrow0,
\]
and that the localized cycle \(T_j\) is obtained from \(S_j\) using at most
\(q_*\) localization cuts. Denote the corresponding mass-minimizing
localization fillings  by
$
\mathcal C_j^1,\ldots,\mathcal C_j^{q_j},$  for $q_j\le q_*,
$
and suppose that
\[
\sum_{\ell=1}^{q_j}
\mathbf{M}(\mathcal C_j^\ell)
\le m_{\mathrm{tot}}.
\]
Then, for every \(\vartheta\in(0,1)\), there exists a finite set
\(W_j\subset\mathbb R^d\) such that
\begin{align*}
\spt(T_j)
\subset  B_{\Lambda_\vartheta h_j}(Y_j&\cup W_j),
\qquad
\Lambda_\vartheta:=1+q_*\vartheta
\\
h_j^k\card{W_j}
&\le
C_{d,k}\vartheta^{-k}m_{\mathrm{tot}}.
\end{align*}
\end{lem}

\begin{proof}
Let
$
K_j^0:=\spt (S_j)$
and inductively, set
$
K_j^\ell
:=
K_j^{\ell-1}\cup\spt (\mathcal C_j^\ell)
$. 
The boundary of the \(\ell\)-th filling is the slice produced by cutting
the current available at the \(\ell\)-th localization step. Hence
\[
\spt (\partial \mathcal C_j^\ell)
\subset K_j^{\ell-1}.
\]
For the selected localized cycle $T_j$, we also have $
\spt (T_j)\subset K_j^{q_j}.$
The goal is to construct finite sets \(W_j^\ell\) iteratively so that the union of balls covers $K_j^\ell$, i.e. 
\[
K_j^\ell
\subset
B_{(1+\ell\vartheta)h_j}(Y_j\cup W_j^1\cup\cdots\cup W_j^\ell).
\]
Fix \(\ell\), and suppose the inclusion holds for $\ell-1$. Decompose the support of the filling into the area close to its boundary
 and the remaining interior.
Denoting the interior by \[
I_j^\ell
:=
\left\{
x\in\spt \mathcal C_j^\ell:
\dist
\bigl(x,\spt (\partial \mathcal C_j^\ell)\bigr)
\ge\vartheta h_j
\right\}
\]
 we obtain the decomposition
\[
\spt (\mathcal C_j^\ell)
=
\left\{
x\in \spt (\mathcal C_j^\ell):
\dist
\bigl(x,\spt (\partial \mathcal C_j^\ell)\bigr)
<\vartheta h_j
\right\}
\cup I_j^\ell,
\]
Since
$
\spt \partial (\mathcal C_j^\ell)
\subset K_j^{\ell-1},
$
the  assumption at $\ell-1$ implies that the boundary layer is covered by
balls of radius
$(1+\ell\vartheta)h_j$
centered at
$
Y_j\cup W_j^1\cup\cdots\cup W_j^{\ell-1}.
$
Choosing \(W_j^\ell\subset I_j^\ell\) to be a maximal
\(\vartheta h_j/2\)-separated set, we obtain
\[
I_j^\ell
\subset
B_{\vartheta h_j/2}(W_j^\ell).
\]
As the set is maximally separated, the balls
$
B_{\vartheta h_j/4}(z),$ are pairwise disjoint  for all $ z\in W_j^\ell$.
Furthermore, each such ball is disjoint from
\(\spt( \partial \mathcal C_j^\ell)\). Since \(\mathcal C_j^\ell\) is
mass minimizing, the monotonicity formula gives
\[
\mathbf{M}
\bigl(
\mathcal C_j^\ell\llcorner B_{\vartheta h_j/4}(z)
\bigr)
\ge
c_{d,k}(\vartheta h_j)^k
\]
and summing over \(z\in W_j^\ell\), we obtain
$
c_{d,k}(\vartheta h_j)^k\card{ W_j^\ell}
\le
\mathbf{M}(\mathcal C_j^\ell),
$ 
or equivalently, 
\[
h_j^k\card{ W_j^\ell}
\le
C_{d,k}\vartheta^{-k}
\mathbf{M}(\mathcal C_j^\ell).
\]
This proves the covering bound at each step. Define
$W_j
:=
\bigcup_{\ell=1}^{q_j}W_j^\ell,$
and
since \(q_j\le q_*\),
\begin{align*}
h_j^k \card{W_j}
&\le
\sum_{\ell=1}^{q_j}h_j^k\card{W_j^\ell}\\
&\le
C_{d,k}\vartheta^{-k}
\sum_{\ell=1}^{q_j}
\mathbf{M}(\mathcal C_j^\ell)\le
C_{d,k}\vartheta^{-k}m_{\mathrm{tot}}.
\end{align*}
Finally using  \(q_j\le q_*\) again, we obtain
\[
\spt (T_j)
\subset
B_{(1+q_*\vartheta)h_j}(Y_j\cup W_j).
\]
completing the proof.
\end{proof}

We now relate these covers to the mass of the limiting
current. Since both the cycles and the covering scales vary with $j$,
the covering asymptotic in Theorem~\ref{thm:covering} does not apply directly.
We therefore show that under global flat convergence with supports
in a common compact set, the normalized covering number
is bounded below by the mass of the limit, with the same sharp
constant $\coverconst$.
\begin{lem}\label{lem:varying-cover}
Let \(T_j,T\in \mathbf I_k(\mathbb R^d;\mathbb Z_2)\) be compactly
supported \(k\)-cycles such that
$
T_j\rightarrow T$
in the global flat norm.
Suppose that
\[
\spt (T)\cup\bigcup_j\spt (T_j)
\subset K
\]
for some compact set \(K\subset\mathbb R^d\). Let \(h_j\rightarrow0\), and
let \(Y_j\subset\mathbb R^d\) be finite sets satisfying
\[
\spt (T_j)
\subset
B_{h_j}(Y_j).
\]
Then $
\liminf_{j\rightarrow\infty}h_j^k\card{Y_j}
\ge
\coverconst\,\mathbf{M}(T).
$
\end{lem}
\begin{proof}
Set $L:=\liminf_{j\rightarrow\infty} h_j^k\card{Y_j},$
and pass to a subsequence along which \(h_j^k \card{Y_j} \rightarrow L\) and 
  assume \(L<\infty\).
We discard every \(y\in Y_j\) for which
\(B_{h_j}(y)\cap\spt (T_j)=\emptyset\) as this only
decreases the number of centers. Since
\(\spt (T_j)\subset K\), the remaining centers belong to
the \(h_j\)-neighborhood of \(K\), and hence, for all sufficiently
large \(j\), to a fixed compact set.
Define the Dirac measures
\[
\nu_j:=h_j^k\sum_{y\in Y_j}\delta_y.
\]
The sum of masses is bounded, so after passing to a further
subsequence,
$
\nu_j\rightharpoonup\nu
$
weakly as Radon measures.
Passing to a further subsequence, we can assume
\[
\sum_{j=1}^{\infty}\flatnorm{T_j-T}<\infty,
\]
and the weak limit $\nu$ is preserved.
It is now enough to prove that for every Borel set $A$,
\[
\nu(A)\geq
\coverconst\mathbf{M}(T\llcorner A).
\]
Since all the measures are supported in a fixed compact set,
\[
\nu(\mathbb R^d)
=
\lim_{j\rightarrow\infty}\nu_j(\mathbb R^d)
=
\lim_{j\rightarrow\infty}h_j^k\card{Y_j}
\leq L.
\]
We prove the inequality by differentiation. Since \(T\) is an
integral current, let \(\lVert T\rVert\) denote the mass measure. For \(\lVert T\rVert\)-almost every \(x\) there
is an approximate tangent \(k\)-plane \(V=\Tan_T (x)\), and
\[
\lim_{r\rightarrow 0}
\frac{\lVert T\rVert(B_r(x))}{\omega_k r^k}=1.
\]
We denote the ball restricted to $V$, 
$
B_r^V(u):=B_r(u)\cap V$ for all $u\in V.$
We may also assume that \(x\) is a differentiation point of \(\nu\)
with respect to \(\lVert T\rVert\). Fix such a point \(x\).
Let \(\proj_V:\mathbb R^d\rightarrow V\) and \(\proj_{V^\perp}:\mathbb R^d\rightarrow V^\perp\) be the linear orthogonal projections, and define
\(\proj_{V,x}(z):=\proj_V(z-x)\). Fix \(\xi\in(0,1/10)\), and first define a cylinder at $x$, and the projected disk (centered at the origin)
\begin{align*}
\Cyl_{\rho,\xi}(x)
&:=
\left\{
z\in\mathbb R^d:
\lVert \proj_{V,x}(z)\rVert<\rho,\ 
\lVert \proj_{V^\perp}(z-x)\rVert<\xi\rho
\right\}\\
D_{\rho,\xi}&:=\overline{B}^V_{(1-\xi)\rho}(0).
\end{align*}
At the multiplicity-one tangent point \(x\), the rescalings of \(T\)
converge locally in the flat norm to one copy of \(V\). Therefore, slicing
 gives, for fixed $\xi$,
\begin{equation}\label{eq:4_4}
\int_{D_{\rho,\xi}}
\flatnorm{
\left\langle
T\llcorner\Cyl_{\rho,\xi}(x),\proj_{V,x},u
\right\rangle
-
[x+u]
}\,d\mathcal H^k(u)
\leq e_{\xi}(\rho)\rho^k,
\end{equation}
where \(e_{\xi}(\rho)\rightarrow0\) as \(\rho\rightarrow 0\).
By the summability of $\flatnorm{T_j-T}$ and the restriction
theorem for Lipschitz sublevel sets, for each fixed $x$ and $\xi$
we have
\[
\flatnorm{(T_j-T)\llcorner \Cyl_{\rho,\xi}(x)}\rightarrow0
\]
for almost every $\rho>0$.
For each such $\rho$, the integrated slicing inequality gives
%
%
\begin{equation}\label{eq:4.5}
\int_{D_{\rho,\xi}}
\flatnorm{
\left\langle
(T_j-T)\llcorner\Cyl_{\rho,\xi}(x),\proj_{V,x},u
\right\rangle
}\,d\mathcal H^k(u)
\rightarrow0.
\end{equation}
Let \(E_{j,\rho,\xi}\subset D_{\rho,\xi}\) be the set of all \(u\in D_{\rho,\xi}\)
for which the zero-dimensional slice
$
\left\langle
T_j\llcorner\Cyl_{\rho,\xi}(x),\proj_{V,x},u
\right\rangle
$
is undefined\footnote{We include  the exceptional null set from the slicing theorem in \(E_{j,\rho,\xi}\).} or has even parity. As the current \([x+u]\) has odd parity, for almost every 
\(u\in E_{j,\rho,\xi}\), their difference also has odd parity.

A compactly supported zero-dimensional $\mathbb{Z}_2$ current with odd
parity has flat norm at least \(1\).  That is, if
$
Z=U+\partial F,
$
then \(\partial F\) has even parity, so \(U\) has odd parity and
\(\mathbf{M}(U)\geq1\). Consequently,
$
\flatnorm{Z}\geq1.
$
Using this observation together with \eqref{eq:4_4}, \eqref{eq:4.5}, and the triangle
inequality, we obtain
\begin{equation}\label{eq:4.6}
\limsup_{j\rightarrow\infty}
\mathcal H^k(E_{j,\rho,\xi})
\leq e_{\xi}(\rho)\rho^k.
\end{equation}

We now state the elementary almost-cover estimate used below.
Suppose that \(D\subset V\) is a closed \(k\)-dimensional ball,
\(r_j\rightarrow0\), and \(A_j\subset V\) are finite sets such that
\[
D\setminus E_j
\subset
B^V_{r_j}(A_j).
\]
We claim that for every \(\eta \in(0,1)\),
\begin{equation}\label{eq:4.7}
\liminf_{j\rightarrow\infty}r_j^k\card{A_j}
\geq
(1+\eta)^{-k}\coverconst\mathcal H^k(D)
-
C_k\eta^{-k}\limsup_{j\rightarrow\infty}\mathcal H^k(E_j).
\end{equation}
To see this, first define
\[
D_j^-:=\{u\in D:\dist(u,V\setminus D)>\eta r_j\}
\]
and consider the portion of \(D_j^-\) not covered by the enlarged
balls \(B^V_{(1+\eta)r_j}(A_j)\). Choose a maximal
\((2\eta r_j)\)-separated subset \(W_j\) of this remaining set.
For every \(w\in W_j\),
\[
B^V_{\eta r_j}(w)\subset E_j.
\]
Otherwise some point of this ball would lie in
\(D\setminus E_j\), and hence be within distance \(r_j\) of a point in
\(A_j\) and hence \(w\) would be within 
\((1+\eta)r_j\) of \(A_j\).
For all \(w\in W_j\), the balls \(B^V_{\eta r_j}(w)\) are pairwise disjoint,
and therefore
\begin{equation}\label{eq:4.8}
\card{W_j}\,\omega_k(\eta r_j)^k
\leq
\mathcal H^k(E_j).
\end{equation}

By maximality, \(B^V_{2\eta r_j}(W_j)\) covers
the remaining part of \(D_j^-\). Finally, the boundary layer
\(D\setminus D_j^-\) can be covered by \(O(1/r_j^{k-1})\) balls of
radius \((1+\eta)r_j\), which is negligible after multiplying by
\(r_j^k\). Since \(2\eta\leq1+\eta\), we obtain a cover
of \(D\) by balls of radius \((1+\eta)r_j\), using the centers in
\(A_j\cup W_j\), and \(o(r_j^{-k})\) additional
centers. 
Since $D$ is closed and convex, we may project all centers of
this cover onto $D$. This projection can only decrease the distance from a center to any point in $D$ and so the projection neither increases the covering radius nor
the number of centers, and the resulting cover has centers within the set.
Thus,
\[
\covnum_{(1+\eta)r_j}(D)
\leq
 \card{A_j}+\card{W_j}+o(r_j^{-k}).
\]
Multiplying by \(r_j^k\), using \eqref{eq:4.8} and applying the Euclidean
covering asymptotic
\[
\lim_{r\rightarrow 0 }r^k \covnum_r (D)
=
\coverconst\mathcal H^k(D)
\]
 from Theorem \ref{thm:covering}, which
proves \eqref{eq:4.7}.
We now apply this estimate to the projected centers. If
\[
u\in D_{\rho,\xi}\setminus E_{j,\rho,\xi},
\]
then the corresponding slice of \(T_j\) is nonzero. Hence the fiber
\[
x+u+V^\perp
\]
meets
\(\spt (T_j)\cap\Cyl_{\rho,\xi}(x)\). Choose an
intersection point \(z\). Since the \(h_j\)-balls centered at \(Y_j\)
cover \(\spt (T_j)\), there exists a \(y\in Y_j\) with
$
\dist(z,y)\leq h_j$.
Because \(\proj_{V,x}(z)=u\), this implies
\[
\lVert \proj_V(y-x)-u\rVert\leq h_j.
\]
For fixed \(\rho\), take \(j\) sufficiently large that
\(h_j\leq\xi\rho\). Since \(z\in\Cyl_{\rho,\xi}(x)\),
\[
\dist(y,x)
\leq \dist(y,z)+\dist(z,x)
\leq
\left(\xi+\sqrt{1+\xi^2}\right)\rho.
\]
Setting
$
R_{\xi}:=\sqrt{1+\xi^2}+\xi,
$
it follows that the projected \(h_j\)-balls whose centers correspond
to
$
Y_j\cap\overline{B}_{R_{\xi\rho}}(x)
$
cover \(D_{\rho,\xi}\setminus E_{j,\rho,\xi}\). Applying \eqref{eq:4.7}
and then \eqref{eq:4.6} gives
\begin{equation}\label{eq:4.9}
\liminf_{j\rightarrow\infty}
h_j^k
\card{Y_j\cap \overline{B}_{R_{\xi\rho}}(x)}
\geq
(1+\eta)^{-k}\coverconst\omega_k(1-\xi)^k\rho^k-C_k\eta^{-k}e_{\xi}(\rho)\rho^k.
\end{equation}
Choose \(\rho\rightarrow 0\) according to radii for which the preceding
restriction statements hold and
$
\nu\bigl(\partial B_{R_{\xi}\rho}(x)\bigr)=0.
$ Weak convergence then gives
\[
\nu\bigl(\overline{B}_{R_{\xi}\rho}(x)\bigr)
=
\lim_{j\rightarrow\infty}
h_j^k\card{Y_j\cap\overline{B}_{R_{\xi}\rho}(x)}.
\]
Therefore, \eqref{eq:4.9} implies
\begin{equation}\label{eq:4.10}
\nu\bigl(\overline{B}_{R_{\xi}\rho}(x)\bigr)
\geq
(1+\eta)^{-k}\coverconst\omega_k(1-\xi)^k\rho^k-C_k\eta^{-k}e_{\xi}(\rho)\rho^k.
\end{equation}
On the other hand, the density 1 property of \(T\) (a $\mathbb{Z}_2$ rectifiable current is  1 a.e. on the support) gives
\[
\lVert T\rVert\bigl(B_{R_{\xi}\rho}(x)\bigr)
=
\omega_kR_{\xi}^k\rho^k+o(\rho^k).
\]
Divide \eqref{eq:4.10} by this expression and let \(\rho\rightarrow 0\).
Since \(e_{\xi}(\rho)\rightarrow0\), we obtain
\[
\frac{d\nu}{d\lVert T\rVert}(x)
\geq
(1+\eta)^{-k}\coverconst
\left(\frac{1-\xi}{R_{\xi}}\right)^k.
\]
Finally, let \(\eta\rightarrow 0\) and then \(\xi\rightarrow 0\).
Because \(R_{\xi}\rightarrow1\),
\[
\frac{d\nu}{d\lVert T\rVert}(x)\geq\coverconst
\]
for \(\lVert T\rVert\)-almost every \(x\). Hence
\[
\nu\geq\coverconst\lVert T\rVert.
\]
Taking total masses now yields
\[
L
\geq \nu(\mathbb R^d)
\geq
\coverconst\lVert T\rVert(\mathbb R^d)
=
\coverconst\mathbf{M}(T),
\]
which proves the lemma.
\end{proof}
The limit obtained above  need not be
polyhedral or admit a pseudomanifold parameterization, so we extend
the sharp inequality in Theorem~\ref{thm:sharp-fillrad-pseudomanifold} to compactly supported integral cycles by polyhedral approximation.
\begin{cor}\label{cor:sharp-fillrad-current}
Every compactly supported integral \(k\)-cycle
$
T\in\mathbf I_k(\mathbb R^d;\mathbb Z_2)
$ satisfies
\begin{equation}\label{eq:sharp-fillrad-current}
\mathbf{M}(T)
\geq
s_k\,\FillRad (T)^k,
\end{equation}
where $s_k$ is the optimal Euclidean isoperimetric constant. 
\end{cor}

\begin{proof}
 The assertion is immediate if \(T=0\), so assume
that \(T\neq0\).
We construct finite polyhedral cycles approximating $T$,  which are homologous to $T$ through chains
supported in shrinking neighborhoods of $K=\spt (T)$ and 
convergence to $\mathbf{M}(T)$. 

We use the following elementary consequence of the deformation
theorem. If \(Z\) is a compactly supported integral \(q\)-cycle,
where \(1\leq q<d\), then there is a filling \(W\) such that
\begin{align*}
\partial W=Z,
\qquad&
\mathbf{M}(W)
\leq C_{d,q}\mathbf{M}(Z)^{(q+1)/q},\\
\spt (W)
&\subset
B_{C_{d,q}\mathbf{M}(Z)^{1/q}}
\bigl(\spt (Z)\bigr).
\end{align*}
Moreover, \(W\) can be chosen polyhedral if \(Z\) is polyhedral.
Apply the deformation theorem on a cubical mesh of size
\(\ell=A_{d,q}\mathbf{M}(Z)^{1/q}\), with
\(A_{d,q}\) sufficiently large. The deformed cycle has volume
strictly less than \(\ell^q\), and hence is zero. This follows from the fact that a nonzero
cellular \(q\)-chain contains at least one  locally constant \(q\)-face.
The deformation theorem \cite[Theorem~1.1]{White1999Deformation} also provides a $(q+1)$-chain $W$
with the stated mass and support bounds and
$\partial W=Z-Z'$. Since $Z'=0$, this gives $\partial W=Z$. For \(Z=0\), we take \(W=0\).
    
Choose \(a_m\rightarrow 0\), with \(a_m<1\), and let
$
K_m:=\{x\in\mathbb R^d:
\dist (x,K)\leq a_m\}.
$
Apply polyhedral approximation in \(K_m\), using
\cite[Theorem~3.4]{MarcheseStuvard2018}.
Since \(\partial T=0\), we obtain finite polyhedral
\(k\)-chains \(A_m\) such that
\[
\mathbf{M}(A_m)
\leq\mathbf{M}(T)+a_m,
\qquad
\mathbf{M}(\partial A_m)\leq a_m.
\]
The estimate for the flat norm in \(K_m\) also gives integral
currents \(U_m\) and \(F_m\), of dimensions \(k\) and \(k+1\),
respectively, with
\begin{align*}
T=A_m+U_m+\partial F_m,
\qquad&
\mathbf{M}(U_m)
+\mathbf{M}(F_m)<2a_m,\\
\spt (A_m)\cup\spt (U_m)
&\cup\spt (F_m)\subset K_m.
\end{align*}

We now remove the boundary of \(A_m\).
If \(k\geq2\), apply the preceding filling construction to the
polyhedral \((k-1)\)-cycle \(\partial A_m\). This gives a
polyhedral \(k\)-chain \(\mathcal C_m\) with
\begin{align*}
\partial \mathcal C_m=\partial A_m,
\qquad&
\mathbf{M}(\mathcal C_m)
\leq C_{d,k}a_m^{k/(k-1)},\\
\spt (\mathcal C_m)
&\subset
B_{C_{d,k}a_m^{1/(k-1)}}(K_m).
\end{align*}
If \(k=1\), then \(\partial A_m=0\), since a 
$\mathbb{Z}_2$ zero-dimensional chain (which is not 0) has mass at least \(1\), whereas
\(\mathbf{M}(\partial A_m)\leq a_m<1\).
In this case, set \(\mathcal C_m=0\).
Define
\[
T_m:=A_m+\mathcal C_m,
\qquad
E_m:=U_m+\mathcal C_m.
\]
Both \(T_m\) and \(E_m\) are cycles, because
\(\partial U_m=\partial A_m=\partial \mathcal C_m\).
Moreover,
\[
T=T_m+E_m+\partial F_m,
\qquad
\mathbf{M}(E_m)\rightarrow0.
\]
Apply the same filling construction to \(E_m\), obtaining
\(W_m\) with \(\partial W_m=E_m\), mass tending to zero,
and support at distance tending to zero from
\(\spt (E_m)\). Setting \(V_m:=F_m+W_m\), we have
\[
T=T_m+\partial V_m,
\qquad
\mathbf{M}(V_m)\rightarrow0,
\]
and, for some \(\delta_m\rightarrow 0\),
\[
\spt (T_m)\cup\spt (V_m)
\subset
B_{\delta_m}(K).
\]
In particular, \(T_m\rightarrow T\) in the global flat norm. Since
\[
\mathbf{M}(T_m)
\leq\mathbf{M}(A_m)+\mathbf{M}(\mathcal C_m)
\leq\mathbf{M}(T)+o(1),
\]
lower semicontinuity gives
\[
\mathbf{M}(T_m)
\rightarrow\mathbf{M}(T).
\]

By arbitrarily small generic perturbations of the vertices, we may
assume that each $T_m$ is in affine general position. Choosing the perturbations
sufficiently small, we can adjust $V_m$
by the corresponding affine homotopy and we retain the following properties:
\begin{align*}
\mathbf{M}(T_m)\rightarrow\mathbf{M}(T),
\qquad T-T_m=\partial V_m,\\
\spt (T_m)\cup\spt (V_m)
\subset B_{\delta_m}(K),
\qquad \delta_m\rightarrow 0.
\end{align*}
We claim that
\begin{equation}\label{eq:polyhedral-fillrad-comparison}
\FillRad (T_m)
\geq
\FillRad (T)-\delta_m.
\end{equation}
Indeed, if \(r>\FillRad (T_m)\), choose a filling
\(W_m\) of \(T_m\) supported in
\(B_r(\spt (T_m))\).
Then
\[
\partial(W_m+V_m)=T,
\qquad
\spt (W_m+V_m)
\subset B _{r+\delta_m}(K).
\]
Consequently, \(\FillRad (T)\leq r+\delta_m\).
Letting \(r\rightarrow \FillRad (T_m)\) proves
\eqref{eq:polyhedral-fillrad-comparison}.

Suppose now, for a contradiction, that
$
\mathbf{M}(T)
<s_k\FillRad (T)^k.$ 
By convergence of mass and
\eqref{eq:polyhedral-fillrad-comparison}, for all sufficiently
large \(m\) we have
\[
\mathbf{M}(T_m)
<s_k\bigl(\FillRad (T)-\delta_m\bigr)^k
\leq s_k\FillRad (T_m)^k.
\]
On the other hand,
Lemma~\ref{lem:pseudomanifold-parametrization} provides a finite
pseudomanifold parameterization of \(T_m\) without increasing
its \(k\)-volume. Therefore
Theorem~\ref{thm:sharp-fillrad-pseudomanifold} gives
\[
\mathbf{M}(T_m)
\geq s_k\FillRad (T_m)^k,
\]
a contradiction. This proves \eqref{eq:sharp-fillrad-current}.
\end{proof}

We now complete the proof of Proposition~\ref{prop:optimal}.

\begin{proof}[Proof of Proposition~\ref{prop:optimal}]
As noted earlier, it is sufficient to prove
\begin{equation}\label{eq:min-points-asymptotic}
\liminf_{\pi\rightarrow\infty}
\pi^{-k}N_\pi^{(k)}
\geq
\coverconst s_k.
\end{equation}
Let \(\pi_j\rightarrow\infty\) be arbitrary, and pass to a subsequence along
which
\[
\pi_j^{-k}N_{\pi_j}^{(k)}
\rightarrow
L
:=
\liminf_{j\rightarrow\infty}
\pi_j^{-k}N_{\pi_j}^{(k)}.
\]
There is nothing to prove if \(L=\infty\), so suppose that
\(L<\infty\).
Choose a finite point set \(X_j\) with
$
\card{X_j}=N_{\pi_j}^{(k)},
$
supporting a persistent \(k\)-class with (multiplicative) persistence at least \(\pi_j\).
Normalize the death time to \(1\) by scaling, and denote its normalized birth time by
\(b_j\). Then $b_j\leq 1/\pi_j$.
After passing to a further subsequence, we may assume that
$
b_j^k\card{X_j}\rightarrow L'
$
for some \(L'\leq L\).

Fix \(\eta>0\), and apply
Lemma~\ref{lem:separated-net-regularization}. Setting
$
h_j:=(1+\eta)b_j,
$
we obtain finite sets \(Y_j\subset X_j\) and compactly supported
polyhedral \(k\)-cycles \(S_j\) such that
\begin{equation}\label{eq:regularized-cycle-properties}
\spt (S_j)
\subset
B_{h_j}(Y_j),
\qquad
\card{Y_j}\leq\card{X_j},
\qquad
\FillRad (S_j)\geq1-o(1),
\end{equation}
and
\begin{equation}\label{eq:regularized-cycle-volume}
\mathbf{M}(S_j)
\leq
C_{\mathrm{net}}(d,k,\eta)h_j^k\card{Y_j}.
\end{equation}
Since
\[
h_j^k \card{Y_j}
\leq
(1+\eta)^k b_j^k\card{X_j},
\]
\eqref{eq:regularized-cycle-volume} gives a uniform bound
$
\mathbf{M}(S_j)\leq M_0.
$
Now fix \(\varepsilon\in(0,1)\) and \(\zeta>0\). Applying
Lemma~\ref{lem:persistent-localized-cycle}, we obtain integral \(k\)-cycles
\(T_j\), points \(x_j\), and a radius \(R_{\mathrm{loc}}<\infty\), independent of
\(j\), such that
\begin{equation}\label{eq:selected-cycle-properties}
\spt (T_j)\subset B_{R_{\mathrm{loc}}}(x_j),
\qquad
\mathbf{M}(T_j)\leq M_0+\zeta,
\qquad
\FillRad (T_j)\geq1-\varepsilon.
\end{equation}
Let \(q_*\) be the uniform bound on the number of localization cuts
in the proof of Lemma~\ref{lem:persistent-localized-cycle}. Each localization
filling \(\mathcal C_j^\ell\) satisfies
\[
\mathbf{M}(\mathcal C_j^\ell)\leq\frac{m_{\mathrm{fill}}}{2},
\]
and the parameters were chosen so that \(q_*m_{\mathrm{fill}}<\zeta\). Hence
\begin{equation}\label{eq:total-filling-volume}
\sum_{\ell=1}^{q_j}
\mathbf{M}(\mathcal C_j^\ell)
\leq
\frac{q_*m_{\mathrm{fill}}}{2}
<
\frac{\zeta}{2}.
\end{equation}
Fix \(\vartheta\in(0,1)\) and set
$
\Lambda_\vartheta:=1+q_*\vartheta$.
Lemma~\ref{claim:cover-localization-fillings}, applied using
\eqref{eq:regularized-cycle-properties} and \eqref{eq:total-filling-volume}, produces a finite set
\(W_j\subset\mathbb R^d\) such that
\begin{align}
\spt (T_j)
&\subset
B_{\Lambda_\vartheta h_j}(Y_j\cup W_j)\label{eq:cycle-cover-with-filling}\\
h_j^k\card{W_j}
&\leq
C_{d,k}\vartheta^{-k}\zeta.\label{eq:filling-covering-cost}
\end{align}
We translate the localized cycles and centers by \(-x_j\) and denote the translations by 
$
\widehat T_j,$
$\widehat Y_j,$ and 
$\widehat{W}_j
$ respectively. 
By \eqref{eq:selected-cycle-properties},
\[
\spt (\widehat T_j)\subset\overline{B}_{R_{\mathrm{loc}}}(0)  ,
\qquad
\mathbf{M}(\widehat T_j)\leq M_0+\zeta.
\]
The Federer--Fleming compactness theorem therefore gives, after
passing to a further subsequence,
\[
\widehat T_j
\rightarrow
T_{\varepsilon,\zeta}
\]
in the global flat norm, for some compactly supported integral
\(k\)-cycle \(T_{\varepsilon,\zeta}\). By
Lemma~\ref{lem:fillrad-semicontinuity},
\begin{equation}\label{eq:limit-cycle-fillrad}
\FillRad (T_{\varepsilon,\zeta})
\geq
\limsup_{j\rightarrow\infty}\FillRad (\widehat T_j)
\geq
1-\varepsilon.
\end{equation}
Applying Corollary~\ref{cor:sharp-fillrad-current} and
\eqref{eq:limit-cycle-fillrad}, we obtain
\begin{equation}\label{eq:limit-cycle-volume}
\mathbf{M}(T_{\varepsilon,\zeta})
\geq
s_k\FillRad (T_{\varepsilon,\zeta})^k
\geq
s_k(1-\varepsilon)^k.
\end{equation}
Define
$
A_j:=
\widehat Y_j\cup\widehat{W}_j.
$
By \eqref{eq:cycle-cover-with-filling}, the balls  with centers in \(A_j\) and radius
$
\widetilde h_j:=\Lambda_\vartheta h_j
$
 cover \(\spt (\widehat T_j)\).
Lemma~\ref{lem:varying-cover} and
\eqref{eq:limit-cycle-volume} therefore give
\[
\liminf_{j\rightarrow\infty}
\widetilde h_j^k\card{A_j}
\geq
\coverconst\mathbf{M}(T_{\varepsilon,\zeta})
\geq
\coverconst s_k(1-\varepsilon)^k.
\]
Since \(\widetilde h_j=\Lambda_\vartheta h_j\),
\begin{equation}\label{eq:cycle-center-lower-bound}
\liminf_{j\rightarrow\infty}
h_j^k\card{A_j}
\geq
\Lambda_\vartheta^{-k}\coverconst s_k(1-\varepsilon)^k.
\end{equation}
Because $
\card{\widehat{Y}_j}
\geq    
\card{A_j}-\card{\widehat{W}_j},
$
\eqref{eq:filling-covering-cost} and
\eqref{eq:cycle-center-lower-bound} imply that 
\begin{equation}\label{eq:original-center-lower-bound}
\liminf_{j\rightarrow\infty}h_j^k\card{ Y_j}
\geq
\Lambda_\vartheta^{-k}\coverconst s_k(1-\varepsilon)^k
-
C_{d,k}\vartheta^{-k}\zeta.
\end{equation}
Using \(Y_j\subset X_j\), \(h_j=(1+\eta)b_j\), and the definition of
\(L'\), we conclude from Equation\eqref{eq:original-center-lower-bound} that
\[
(1+\eta)^kL'
\geq
\Lambda_\vartheta^{-k}\coverconst s_k(1-\varepsilon)^k
-
C_{d,k}\vartheta^{-k}\zeta.
\]
First let \(\zeta\rightarrow 0\), and then let \(\vartheta\rightarrow 0\).
Since \(\Lambda_\vartheta=1+q_*\vartheta\), this yields
\[
L'
\geq
(1+\eta)^{-k}\coverconst s_k(1-\varepsilon)^k.
\]
Finally, let \(\varepsilon\rightarrow 0\) and then
\(\eta\rightarrow 0\). We obtain
$
L'\geq\coverconst s_k.$
Since \(L'\leq L\), this proves
\eqref{eq:min-points-asymptotic}.
Applying Theorem \ref{thm:covering} to \(\mathbb{S}^k\) gives
\[
\lim_{\pi\rightarrow\infty}
\pi^{-k}N_{1/\pi}^{\mathrm{cov}}(\mathbb{S}^k)
=
\coverconst s_k.
\]
Consequently,
\[
\liminf_{\pi\rightarrow\infty}
\frac{N_\pi^{(k)}}
     {N_{1/\pi}^{\mathrm{cov}}(\mathbb{S}^k)}
\geq1.
\]
Therefore, for every \(\delta>0\) and all sufficiently large \(\pi\),
\[
N_\pi^{(k)}
\geq
(1-\delta)N_{1/\pi}^{\mathrm{cov}}(\mathbb S^k).
\]
Proposition~\ref{prop:optimal} follows immediately from the above and the covering asymptotic in
 Theorem \ref{thm:covering}.
\end{proof}

\section*{Statement of AI Use}
Generative AI tools, including ChatGPT (5.6 and 6), Claude Opus and Fable, were used during the preparation of this manuscript to support proofreading, the checking of mathematical arguments, and the identification of typographical and grammatical errors, as well as to assist with reference searches and discussion of selected aspects of the literature. The manuscript was written by the authors; these tools served solely as aids in its preparation and verification. All mathematical content, arguments, references, and suggested corrections were independently reviewed and verified by the authors, who take full responsibility for the final content of the paper.

\section*{Acknowledgment}
The authors are grateful to Huy Nguyen for very  helpful discussions, and also to Parker Duncan, for his part in the preliminary steps of this research. The authors were partially supported by the Leverhulme Trust grant RPG-2023-144, and by the EPSRC grant EP/Y028872/1. 
OB was partially supported by the EPSRC grant EP/Y008642/1.

\appendix
\section{Details of the extension of Theorem~\ref{thm:sharp-fillrad-pseudomanifold} to finite polyhedral cycles}\label{app:gromov_proof}

In this appendix, we give the polyhedral version of Gromov's visual-hull argument and
deduce the sharp filling-radius inequality for simplex-wise affine maps
of closed pseudomanifolds. This formulation allows self-intersections
of the parameterized cycle. All chains, homology groups, degrees, and
linking numbers in this appendix have coefficients in $\mathbb{Z}_2$.
Throughout, \(1\leq k<d\), and
$s_k:=\Vol_k(\S^k)=\mathcal H^k(\S^k).$
We note that here  as in Section \ref{sec:optimal}, $B_r$ denotes open balls.

 We first recall Gromov’s spherical sweepout inequality~\cite[Section~8.1.B$'$, pp.~105--106]{gromovFillingRiemannianManifolds1983}.. Together with Alexander duality, this inequality will show that the cycle admits a filling supported in every open neighborhood of its visual hull, which we define below.

\begin{lem}
\label{lem:app-spherical-sweepout}
Let \(P\) and \(Q\) be finite closed pseudomanifolds of dimensions
\(k\) and \(q\geq0\), respectively, and let
$
F:P\times Q\longrightarrow \S^{k+q}$
be a continuous, piecewise smooth map of nonzero degree over $\mathbb{Z}_2$.
Writing \(F_a(p):=F(p,a)\), we have
\[
\sup_{a\in Q}\int_P J_kF_a\,d\mathcal H^k\geq s_k.
\]
The integrals are computed on the top-dimensional simplices of \(P\)
and count parameterization multiplicity.
\end{lem}
The pseudomanifold formulation is part of the argument in
\cite[Section~8.1, Fact~1, p.~106]{gromovFillingRiemannianManifolds1983}. For \(q>0\),
the input is Almgren's min--max theory for cycles, applied here with
$\mathbb{Z}_2$ coefficients, together with the lower bound \(s_k\) for
the mass of a nonzero stationary integral \(k\)-varifold in the unit
sphere. The min--max conclusion bounds this mass above by the
supremum of the slice volumes. For \(q=0\), at least one slice has
nonzero degree, and the conclusion follows directly from the area
formula. Using this we can show the required inequality.

\begin{thm}
\label{thm:app-polyhedral-fillrad}
Every finite polyhedral \(k\)-cycle \(T\) in \(\R^d\)
satisfies
\begin{equation*}
\FillRad(T)
\leq\left(\frac{\Vol_k(T)}{s_k}\right)^{1/k}.
\end{equation*}
\end{thm}

\begin{proof}
Assume $T\neq0$, and set $K:=\spt(T)$.
Choose a geometric simplicial representation of $T$.
The construction in the proof of
Lemma~\ref{lem:pseudomanifold-parametrization} yields a
finite closed simplicial pseudomanifold $P$ and a
simplex-wise affine map $\varphi:P\rightarrow K$ such that
\begin{equation}\label{eq:app-mass-parameterization}
\varphi_\#[P]=T,\qquad \varphi(P)=K,\qquad
\int_P J_k(\varphi)\,d\mathcal H^k=\Vol_k(T).
\end{equation}
The volume equality follows because distinct geometric
$k$-simplices have disjoint relative interiors.

For \(x\notin K\), let
\[
r_x(y):=\frac{y-x}{|y-x|}\in \S^{d-1}
\]
be radial projection, and define the visual volume (see \cite[Section~8.1, p.~103]{gromovFillingRiemannianManifolds1983})
\[
\mathrm{Vis}_T(x):=\frac1{s_k}\int_P J_k(r_x\circ \varphi)\,d\mathcal H^k.
\]
where $J_k$ is the Jacobian taken simplex by simplex. Since the
lower-dimensional faces have zero \(k\)-dimensional measure,
this definition does not require differentiability across them.
For \(u=(y-x)/|y-x|\),
\[
Dr_x(y)=\frac1{|y-x|}(I-u\otimes u).
\]
The factor \(I-u\otimes u\) is an orthogonal projection, so
\[
J_k(r_x\circ \varphi)(p)
\leq\frac{J_k(\varphi(p))}{|\varphi(p)-x|^k}.
\]
Together with Equation \eqref{eq:app-mass-parameterization}, this gives
\begin{equation}\label{eq:app-visual-estimate}
\mathrm{Vis}_T(x)\leq
\frac{\Vol_k(T)}{s_k\dist(x,K)^k}.
\end{equation}
Equivalently, if \(L_y\) is the tangent \(k\)-plane at an interior
point of a geometric simplex, the visual volume can be written as
\[
\mathrm{Vis}_T(x)=\frac1{s_k}\int_K
\frac{|\proj_{L_y^\perp}(y-x)|}{|y-x|^{k+1}}
\,d\mathcal H^k(y).
\]

Define the visual hull (see \cite[Definition~8.1.A, p.~104]{gromovFillingRiemannianManifolds1983}) as 
\begin{equation*}
\mathrm{Hull}(K):=K\cup\{x\in\R^d\setminus K \;| \;\mathrm{Vis}_T(x)\geq1\}.
\end{equation*}
We include $K$ explicitly and do not require a continuous extension
of $\mathrm{Vis}_T$ across its singular points. The function $\mathrm{Vis}_T$ is
continuous on $\R^d\setminus K$. On each of the 
 simplices, the integrand depends continuously on $x$, with
a uniform integrable bound on compact subsets of the complement.
By Equation \eqref{eq:app-visual-estimate}, $\mathrm{Vis}_T(x)\rightarrow 0$ as
$|x|\rightarrow\infty$. Thus $\mathrm{Hull}(K)$ is compact, and, with
\(R:=(\Vol_k(T)/s_k)^{1/k}\),
\begin{equation}\label{eq:app-hull-neighborhood}
\mathrm{Hull}(K)\subset\{x | \dist(x,K)\leq R\}.
\end{equation}

We claim that $T$ bounds in every open neighborhood of $\mathrm{Hull}(K)$.
Let \(U\supset \mathrm{Hull}(K)\) be open. Choose a compact finite polyhedral
neighborhood $L$ with
\[
\mathrm{Hull}(K)\subset\operatorname{int}(L),\qquad L\subset U,
\]
and triangulate \(L\) such that \(K\) is a subcomplex.
Suppose that the reduced homolgy class of \(T\) is nonzero in
\(\tilde{\Hg}_k(L)\). Pairing this class with cohomology and
applying Alexander duality gives a finite \((d-k-1)\)-cycle
in \(\R^d\setminus L\) such that its linking number with
$T$ is 1. 
%
Represent the complementary cycle by a piecewise affine map
\[
g:Q^{d-k-1}\rightarrow\R^d\setminus L
\]
from a finite closed pseudomanifold, using the same face-pairing
construction. The Gauss map
\[
F:P\times Q\longrightarrow \S^{d-1},\qquad
F(p,a):=\frac{\varphi(p)-\varphi(a)}{|\varphi(p)-\varphi(a)|},
\]
has degree equal to the linking number of the two
represented cycles (over $\mathbb{Z}_2$). Hence the degree of $F$ is 1 and
Lemma~\ref{lem:app-spherical-sweepout} implies that
\begin{equation}\label{eq:app-sweepout-contradiction}
s_k\leq\sup_{a\in Q}\int_P J_k(r_{g(a)}\circ \varphi)
\,d\mathcal H^k
=s_k\sup_{a\in Q}\mathrm{Vis}_T(g(a)).
\end{equation}
On the other hand, \(g(Q)\cap \mathrm{Hull}(K)=\emptyset\), so
\(\mathrm{Vis}_T(g(a))<1\) for every \(a\in Q\). Compactness of \(Q\)
and continuity of \(\mathrm{Vis}_T\circ g\) give
\(\sup_{a\in Q}\mathrm{Vis}_T(g(a))<1\), contradicting
 Equation \eqref{eq:app-sweepout-contradiction}.

It follows that \([T]=0\) in \(\tilde{\Hg}_k(L)\).
Since \(L\) is a finite polyhedron, there is a 
a finite polyhedral \((k+1)\)-chain \(W\), after subdivision if
necessary, with
\[
\partial W=T,\qquad\spt(W)\subset L\subset U.
\]
In particular, \(W\) 
is an admissible filling of
$T$.
For each $\varepsilon>0$,
Equation \eqref{eq:app-hull-neighborhood} allows us to take
\(U=B_{R+\varepsilon}(K)\), so
\(\FillRad(T)\leq R+\varepsilon\). Letting
$\varepsilon\rightarrow0$ proves the required statement.
\end{proof}

\begin{cor}
\label{cor:app-mapped-pseudomanifold}
Let $M$ be a finite, possibly disconnected, closed simplicial
$k$-pseudomanifold and $\varphi:M\rightarrow\R^d$ be
simplex-wise affine. Set $T:=\varphi_\#[M]$ and
$K:=\spt(T)$, and write
\[
\Vol_k(\varphi):=\int_M J_k(\varphi)\,d\mathcal H^k.
\]
Then
\begin{equation}\label{eq:app-mapped-fillrad}
\FillRad(T)
\leq\left(\frac{\Vol_k(K)}{s_k}\right)^{1/k}
\leq\left(\frac{\Vol_k(\varphi)}{s_k}\right)^{1/k}.
\end{equation}
\end{cor}
\begin{proof}
The pushforward $T$ is a finite polyhedral cycle, since
$\partial T=\varphi_\#\partial[M]=0$.
Moreover, since $K\subset\varphi(M)$ as a set, together with the area formula implies
\[
\Vol_k(K)
\leq\Vol_k(\varphi(M))
\leq\Vol_k(\varphi).
\]
The result them follows from
Theorem~\ref{thm:app-polyhedral-fillrad}.
\end{proof}
For the parameterizations of simplicial cycles in affine general
position such as \v Cech cycles on random points, distinct top-dimensional simplices
overlap only in sets of zero \(k\)-dimensional measure. Consequently,
\[
\Vol_k(\varphi_\#[M])
=\Vol_k(\varphi)
=\mathcal H^k(\varphi(M)),
\]
and Equation \eqref{eq:app-mapped-fillrad} gives the asserted inequality with
the geometric \(k\)-volume. Furthermore, the constant is sharp as a round \(k\)-sphere of radius
$r$ has filling radius of $r$ and volume \(s_k r^k\) and
polyhedral approximations give the same optimal constant for
Theorem~\ref{thm:app-polyhedral-fillrad}.
\bibliographystyle{plain}
\bibliography{zotero}
 \end{document}